\documentclass[a4paper]{amsart}
\usepackage[T1]{fontenc}
\usepackage{lmodern}
\usepackage{amssymb}
\usepackage[all]{xy}
\usepackage{enumitem}
\usepackage{nicefrac,stmaryrd}
\usepackage{mathtools}
\usepackage[british]{babel}
\usepackage{microtype}
 \usepackage{tikz-cd, tikz}
\usepackage[pdftitle={Topological correspondences-II},
pdfauthor={Rohit Dilip Holkar},
  pdfsubject={}
]{hyperref}
\usepackage[lite]{amsrefs}
 \usepackage{amsthm, amsfonts}
 \usepackage{nicefrac, color}
\usepackage{caption}

\numberwithin{equation}{section}
\theoremstyle{plain}
\newtheorem{theorem}[equation]{Theorem}
\newtheorem{lemma}[equation]{Lemma}
\newtheorem{proposition}[equation]{Proposition}

\theoremstyle{definition}
\newtheorem{definition}[equation]{Definition}

\newtheorem{observation}[equation]{Observation}

\theoremstyle{remark}
\newtheorem{remark}[equation]{Remark}

\newtheorem{example}[equation]{Example}

\newcommand{\co}{\colon}
\newcommand*{\defeq}{\mathrel{\vcentcolon=}}

\newcommand{\R}{\mathbb R}
\newcommand{\C}{\mathbb C}

\newcommand*{\base}[1][H]{{#1}^{(0)}}

\newcommand{\nb}{\nobreakdash} 
\newcommand{\7}{\backslash}
\newcommand{\inverse}{^{-1}}
\newcommand{\supp}{\textup{supp}}
\newcommand{\HH}{\mathrm H}
\newcommand{\CC}{\mathrm C}
\newcommand{\homeo}{\approx}
\newcommand{\iso}{\simeq}

\newcommand{\Cont}{\mathrm{C}}
\newcommand{\Contc}{\mathrm{C_c}}

\newcommand{\inpro}[2]{\left\langle#1 \mathbin, #2 \right\rangle}

\newcommand*{\Star}{*\nb-}

\newcommand*{\Bound}{\mathbb B}
\newcommand*{\e}{\mathrm e}
\newcommand*{\dd}{\textup{d}}
\newcommand*{\Id}{\textup{Id}}
\newcommand*{\Cst}{\textup C^*}
\newcommand*{\Hils}[1][H]{\mathcal{#1}}
\newcommand{\Ltwo}{\mathcal L^2}

\title{Composition of topological correspondences}
\author{Rohit Dilip Holkar}
\email{rohit.d.holkar@gmail.com}
\address{Department of Mathematics, Federal University of Santa Catarina, 88.\,040-900, Florianop\'olis, SC, Brazil}
\thanks{}

\begin{document}

\maketitle{}

\begin{abstract}
In~\cite{Holkar2015Construction-of-Corr}, we define a topological correspondence from a locally compact groupoid equipped with a Haar system to another one. In~\cite{Holkar2015Construction-of-Corr}, we show that a topological correspondence, $(X,\lambda)$, from a locally compact groupoid with a Haar system $(G,\alpha)$ to another one, $(H,\beta)$, produces a $\Cst$\nb-correspondence $\Hils(X)$ from $\Cst(G,\alpha)$ to $\Cst(H,\beta)$. In the present article, we describe how to form a composite of two topological correspondences when the bispaces are Hausdorff and second countable in addition to being locally compact.
\end{abstract}

\tableofcontents{}

 \section*{Introduction}

Let $(G,\alpha)$ and $(H,\beta)$ be locally compact groupoids with Haar systems. A topological correspondence from $(G,\alpha)$ to $(H,\beta)$ is a $G$-$H$\nb-bispace $X$ which is equipped with a continuous family of measures $\lambda$ along the momentum map $s_X\colon X\to \base$, and the action of $H$ and the family of measures satisfy certain conditions (See~\cite{Holkar2015Construction-of-Corr}*{Definition 2.1}). We need that the action of $H$ is proper, and the condition on $\lambda$ is that it is $H$\nb-invariant and each measure in $\lambda$ is $(G,\alpha)$\nb-quasi-invariant. The groupoids $G$ and $H$, and the space $X$ are locally compact. Recall the definitions from~\cite{Holkar2015Construction-of-Corr}: we call a subset $A\subseteq X$ of a topological space $X$ \emph{quasi-compact} if every open cover of $A$ has a finite subcover, and $A$ is called \emph{compact} if it is quasi-compact and Hausdorff. The space $X$ is called locally compact if every point $x\in X$ has a locally compact neighbourhood. We call a topological groupoid $G$ locally compact if $G$ is a locally compact topological space and $\base[G]\subseteq G$ is Hausdorff.

The main result in~\cite{Holkar2015Construction-of-Corr} says that a topological correspondence $(X,\lambda)$ from $(G,\alpha)$ to $(H,\beta)$ produces a $\Cst$\nb-correspondence $\Hils(X)$ from $\Cst(G,\alpha)$ to $\Cst(H,\beta)$. Section 3 of~\cite{Holkar2015Construction-of-Corr} discusses many examples of topological correspondences.

Two $\Cst$\nb-correspondences, $\Hils[K]\colon A\to B$ and $\Hils[F]\colon B\to C$, may be composed to get a correspondence $\Hils[K]\mathbin{\hat\otimes}_{B}\Hils[F]\colon A\to C$. On the similar lines, consider two topological correspondences $(X,\alpha)$ and $(Y,\beta)$ from $(G_1,\chi_1)$ to $(G_2,\chi_2)$ and $(G_2,\chi_2)$ to $(G_3,\chi_3)$, respectively. We describe the composite $(Y,\beta)\circ(X,\alpha)\colon(G_1,\chi_1)\to(G_3,\chi_3)$ when $X$ and $Y$ are Hausdorff and second countable in addition to being locally compact. In fact, our construction works when $X$ and and $Y$ are Hausdorff and the space $(X\times_{s_X,\base[G_2],r_Y} Y)/G_2$ is paracompact; here $s_X$ and $r_X$ denote the momentum maps for the actions of $G_2$ on $X$ and $Y$, respectively. And the quotient is taken for the diagonal action of $G_2$ on $X\times_{s_X,\base[G_2],r_Y} Y$.

The composite $(Y,\beta)\circ(X,\alpha)$ should be a pair $(\Omega,\mu)$ where $\Omega$ is a $G_1$-$G_3$\nb-bispace, $\mu$ is a continuous family of measures along the momentum map $s_\Omega\colon \Omega\to\base[G_3]$ and the conditions in~\cite{Holkar2015Construction-of-Corr}*{Definition 2.1} are satisfied. Furthermore, we must have an isomorphism
$\Hils(\Omega)\iso\Hils(X)\hat\otimes_{\Cst(G_2,\chi_2)}\Hils(Y)$ of $\Cst$\nb-correspondences. 

The construction of $\Omega$ is well-known--- it is the quotient space $(X\times_{s_X,\base[G_2],r_Y} Y)/G_2$ for diagonal action of $G_2$ on $X\times_{s_X,\base[G_2],r_Y} Y$. The diagonal action is proper, since the action of $G_2$ on $X$ is proper. Thus the quotient space inherits all the \emph{nice} properties of the fibre product such as Hausdorffness. The harder task is to get the continuous family of measures $\mu$ satisfying the required conditions. 

We need that $\mu\defeq\{\mu_u\}_{u\in\base[G_3]}$ is $G_3$\nb-invariant and each $\mu_u$ is $(G_1,\chi_1)$\nb-quasi-invariant. We explain how to get \emph{one} such family of measures. The reason to write `one such family of measures' is that the family is not unique; it depends on the choice of a certain continuous function on $X\times_{s_X,\base[G_2],r_Y} Y$. However, for any two such families of measures the corresponding $\Cst$\nb-correspondences are naturally isomorphic to $\Hils(X)\mathbin{\hat\otimes}_{\Cst(G_2,\chi_2)}\Hils(Y)$. 

The construction of $\mu$ is one of the most technical part of this article. To explain the problem, motivation and idea of constructing the composite of families of measures, we have to do a computation and discuss some technical ideas. Denote the space $X\times_{s_X,\base[G_2],r_Y} Y$ by $Z$. Then $Z$ carries a $G_3$\nb-invariant continuous family of measures $m\defeq \{m_u\}_{u \in {\base[G]}_3}$ which is given by
\begin{equation}
  \label{eq:prel-1}
\int_{Z} f \;\dd m_u = \int_Y\int_X f(x,y) \; \dd\alpha_{r_Y(y)}(x)\; \dd\beta_{u}(y)  
\end{equation}
for $f\in \Contc(Z)$. Let $\pi\colon Z\to\Omega$ be the quotient map and $\lambda$ be the continuous family of measures along it defined as
\begin{equation}
  \label{eq:prel-2}
  \int_{\pi\inverse([x,y])} f\; \dd\lambda^{[x,y]} \defeq \int_{G_2^{r_Y(y)}} f(x\gamma, \gamma\inverse y) \; \dd\chi_2^{r_Y(y)}(\gamma).
\end{equation}
for $f \in \Contc(Z)$, and $[x,y]\in\Omega$ which is the equivalence class of $(x,y)\in Z$. A very natural choice for $\mu$ is that it is the family of measures on $\Omega$ which gives the disintegration $m=\mu\circ\lambda$. Furthermore, one may expect that the isomorphism $\Hils(X)\hat\otimes_{\Cst(G_2,\chi_2)}\Hils(Y)\iso\Hils(\Omega)$ is induced by the map $\Psi\colon \Contc(Z)\to \Contc(\Omega)$ where $\Psi(F)([x,y])= \int_{G_2}F(x\gamma,\gamma\inverse y)\,\dd\chi^{s_X(x)}_2(\gamma)$. To be more explicit, we view $\Contc(Z)$ and $\Contc(\Omega)$ as pre\nb-Hilbert $\Cst(G_3,\chi_3)$\nb-modules which complete to the Hilbert $\Cst(G_3,\chi_3)$\nb-modules $\Hils(X)\hat\otimes\Hils(Y)$ and $\Hils(\Omega)$. And the map $\Psi\colon \Contc(Z)\to\Cont(\Omega)$ is expected to induce the required isomorphism of Hilbert $\Cst(G_3,\chi_3)$\nb-modules which also gives the desired isomorphism of $\Cst$\nb-correspondences.

However, that is not \emph{exactly} the case. Consider the following example: let $G$ be a group and $H$ a closed proper subgroups of $G$. Let $\alpha$ and $\kappa$ be the Haar measures on $G$ and $H$, respectively. Then $(G,\alpha\inverse)$ is a topological correspondence from $(G,\alpha)$ to $(H,\kappa)$ which is called the \emph{induction correspondence} in~\cite{Holkar2015Construction-of-Corr}*{Example 3.13}. The constant function 1 is the adjoining function for this correspondence. 

Let $X$ be a left $H$\nb-space carrying an $(H,\kappa)$\nb-quasi-invariant measure $\beta$. Let $\Delta_X$ denote the 1\nb-cocycle on the transformation groupoid $H\ltimes X$ that gives the quasi-invariance. Assume that $\Delta_X$ is continuous. Then $(X,\beta)$ is a topological correspondence from $(H,\kappa)$ to the trivial group \textup{Pt}, see~\cite{Holkar2015Construction-of-Corr}*{Example 3.6}. The adjoining function of this correspondence is $\Delta_X$. Furthermore, $\Hils(X)=\Ltwo(X,\beta)$ and the action of $H$ induces the representation of $\Cst(H)$ on $\Ltwo(X,\beta)$. Thus we have 
\[
(G,\alpha) \xrightarrow{(G,\alpha\inverse)} (K,\kappa) \xrightarrow{(X,\beta)} \textup{Pt}.
\]
Let $Z,\Omega,\pi, m,\lambda$ and $\mu$ have the similar meaning as in the above discussion. Then in this situation, $Z=G\times X,\, \Omega=(G\times X)/K,\, m=\alpha\inverse \times \beta$ and $\pi\colon G\times X\to (G\times X)/K$ the quotient map. For $f\in \Contc(Z)$ Equation~\eqref{eq:prel-2} now reads
\[
\int_{\pi\inverse([\gamma,x])} f\; \dd\lambda^{[\gamma,x]}=\int_H f(\gamma \eta,\eta\inverse x)\,\dd\kappa(\eta).
\]
What are the necessary and sufficient conditions get a measure $\mu$ on $(G\times X)/K$ satisfying $\alpha\inverse \times \beta=m=\mu\circ\lambda$?

 We may draw a square as in Figure~\ref{fig:pushig-measure-down} comprising of the spaces, maps and measures discussed above. And then (i) of Proposition~\ref{prop:pushing_measure} implies that if there is such a measures $\mu$, then the equality $m\circ\kappa=m\circ\kappa\inverse$ must hold--- this is the necessary condition. Recall from the discussion above that $m=\alpha\inverse\times\beta$. Thus we must have $(\alpha\inverse\times \beta)\circ \kappa=(\alpha\inverse\times \beta)\circ \kappa\inverse$ on $Z\times K=G\times X\times K$. Let $f\in \Contc(G\times X\times K)$, then a direct computation gives that 
\[
(\alpha\inverse\times \beta)\circ \kappa (f)=\int_G\int_X\int_K f(\gamma,x,\eta)\,\dd\kappa(\eta)\,\dd\beta(x)\,\dd\alpha\inverse(\gamma).
\]
On the other hand,
\[
(\alpha\inverse\times \beta)\circ \kappa\inverse (f)=\int_G\int_X\int_K f(\gamma,x,\eta\inverse)\,\dd\kappa(\eta)\,\dd\beta(x)\,\dd\alpha\inverse(\gamma).
\]
Now (i) first apply Fubini's theorem to $\dd\kappa\,\dd\beta$, (ii) then change the variable $(\gamma,x,\eta\inverse)$ to $(\gamma\eta, \eta\inverse x,\eta)$, (iii) then use the $(H,\kappa)$\nb-quasi\nb-invariance of $\beta$ and the right invariance of $\alpha\inverse$ and (iv) finally apply Fubini's theorem to $\dd\beta\,\dd\kappa$  
to see that the last term equals
 \[
\int_G\int_X\int_K f(\gamma,\eta\inverse x,\eta)\,\Delta_X(\eta,\eta\inverse x)\,\dd\kappa(\eta)\,\dd\beta(x)\,\dd\alpha\inverse(\gamma).
\]
Thus the 1\nb-cocycle $\Delta_X$ on the transformation groupoid $H\ltimes X$ is the obstruction for the measures $(\alpha\inverse\times \beta)\circ \kappa$ and $(\alpha\inverse\times \beta)\circ \kappa\inverse$ to be equal. {(ii)}~of Proposition~\ref{prop:pushing_measure} says that equality  of these measures,  $m\circ\kappa=m\circ\kappa\inverse$, is also a sufficient condition in our situation for the measure $\mu$ to exist.

A similar problem appears in the general setting; there the cocycle $\Delta_X$ is replace by the adjoining function of the second correspondence involved in the composition. How to overcome this obstruction?

Let $(X,\alpha)\colon(G_1,\chi_1)\to(G_2,\chi_2)$ and $(Y,\beta)\colon(G_2,\chi_2)\to(G_3,\chi_3)$ be topological correspondences and let $\Delta_2$ be the adjoining function of $(Y,\beta)$. Then we realise  $\Delta_2$ as a 1\nb-cocycle on the proper groupoid $(Z\rtimes G_2)$ and \emph{decompose} it into a quotient $\Delta_2=b\circ s_{Z\rtimes G_2}/b \circ r_{Z\rtimes G_2}$ for a 0\nb-cochain $b$ on $Z\rtimes G_2$. Here $s_{Z\rtimes G_2}$ and $r_{Z\rtimes G_2}$ denote the source and the range maps of $Z\rtimes G_2$, respectively. Using Proposition~\ref{prop:pushing_measure} we show that there is a unique measure $\mu$ which gives the disintegration $bm=\mu\circ\lambda$. We modify the map $\Psi\colon \Contc(Z)\to\Contc(\Omega)$ discussed above (the discussion following Equation~\eqref{eq:prel-2} on page~\pageref{eq:prel-2}) to consider the 0\nb-cochain $b$. Then this $\mu$ and modified $\Psi$  produce the desired isomorphism of $\Cst$\nb-correspondences. 

In this construction, the required 0\nb-cochain $b$, which is a function of $Z$, is not unique. However, as mentioned earlier, for any two such 0\nb-cochains the $\Cst$\nb-correspondences associated with the composites are isomorphic to $\Hils(X)\mathbin{\hat\otimes}_{\Cst(G_2,\chi_2)}\Hils(Y)$. Given two 0\nb-cochains $b$ and $b'$ which decompose $\Delta_2$ as above, with some slight work, one may show that there is a positive continuous function $c$ on $Z$ with $b=c b'$. This explains the isomorphism of $\Cst$\nb-correspondences associated with the two composites.
\smallskip

 The theory of groupoid cohomology which we need is developed in~\cite{Holkar2015Construction-of-Corr}*{Section 1} and we prove the other necessary results in this article. The main result we need is Proposition~\ref{prop:proper_gpd_first_cohom_zero} which asserts that the first $\R$\nb-valued cohomology of a proper groupoid is trivial. Lemma~\ref{lemma:proper-gpd-has-prob-measures} is the main tool to prove this proposition. This lemma says that a locally compact proper groupoid $G$ equipped with a Haar system and $G\7\base[G]$ paracompact carries an invariant family of probability measures.

Anantharaman Delaroche and Renault introduced the notions of amenability for a \emph{measured} Borel groupoid (\cite{Anantharaman-Renault2000Amenable-gpd}*{Definition 3.2.8}), \emph{measurewise amenability} for a Borel groupoid (\cite{Anantharaman-Renault2000Amenable-gpd}*{Definition 3.3.1}) and \emph{topological} amenability for a locally compact topological groupoid (\cite{Anantharaman-Renault2000Amenable-gpd}*{Definition 2.2.7}). In~\cite{Anantharaman-Renault2000Amenable-gpd}*{Proposition 3.3.5}, they prove that topological amenability implies measurewise amenability. The definition of measurewise amenability implies that if $G$ is measurewise amenable, then for any Borel Haar system on $G$ and a quasi-invariant measure for the Haar system, $G$ is an amenable measured groupoid.

Our Lemma~\ref{lemma:proper-gpd-has-prob-measures} implies that a locally compact proper groupoid with a Haar system $(G,\alpha)$ and $G/\base[G]$ paracompact is topologically amenable. Anantharaman Delaroche and Renault prove this fact---~\cite{Anantharaman-Renault2000Amenable-gpd}*{Proposition 3.3.5} says that topological amenability implies measurewise amenability. \cite[Proposition 2.2.5]{Anantharaman-Renault2000Amenable-gpd} implies Lemma~\ref{lemma:proper-gpd-has-prob-measures} also. But the proposition proves a general statement than the lemma, and both proofs are different. Our proof is a simple minded one.
\smallskip

As shown in~\cite{Holkar2015Construction-of-Corr}*{Example 3.10 and 3.7}, the generalized morphism defined by Buneci and Stachura in~\cite{Buneci-Stachura2005Morphisms-of-lc-gpd}, and the topological correspondences for the groupoids with Hausdorff space of units introduced by Tu in~\cite{Tu2004NonHausdorff-gpd-proper-actions-and-K}, respectively, are topological correspondences. Our construction of composite matches the ones described by Tu~\cite{Tu2004NonHausdorff-gpd-proper-actions-and-K}, and Buneci and Stachura~\cite{Buneci-Stachura2005Morphisms-of-lc-gpd}, see Examples~\ref{exm:Stadler-Ouchi-correspondence-2} and~\ref{exa:Buneci-Stachura-exm}, respectively.
\medskip

We discuss some examples at the end of the article. A continuous map $f\co X\to Y$ of spaces gives a topological correspondence $(X,\delta_X)$ from $Y$ to $X$ (\cite{Holkar2015Construction-of-Corr}*{Example 3.1}). Here $\delta_X=\{\delta_x\}_{x\in X}$ is the familiy of measures where each $\delta_x$ is the point mass at $x\in X$. A continuous group homomorphism $\phi\co G\to H$ gives a topological correspondences $(H, \beta\inverse)$ from $G$ to $H$ (\cite{Holkar2015Construction-of-Corr}*{Example 3.4}) where $\beta$ is the Haar measure on $H$. If $g\colon Y\to Z$ is a function, then Example~\ref{exa:cont-function-as-corr-2} shows that the composite of the topological correspondences $(X,\delta_X)$ and $(Y, \delta_Z)$ is the correspondence obtained from the function $g\circ f\colon X\to Z$. Example~\ref{exa:gp-homo-as-corr} shows that a similar result holds for group homomorphisms. Both these examples agree with the well-known behaviour of the $\Cst$\nb-functor for spaces and groups.

Let $G$ be a locally compact group, and let $H$ and $K$ be closed subgroups of $G$. Example~\ref{exa:topological-induction} shows how to use topological correspondences to induce a \emph{topological} representation of $K$ to $H$. 

Most of our terminology, definitions, hypotheses and notation are defined in~\cite{Holkar2015Construction-of-Corr}. Now we describe the structure of the article briefly. In the first section, we revise few definitions, notation and results in~\cite{Holkar2015Construction-of-Corr}. We prove that every locally compact proper groupoid equipped with a Haar system carries an invariant continuous family of probability measures. Then using this result we prove that the first cohomology of a proper groupoid is trivial. 

In the second section, we describe the composition of topological correspondences and prove the main result Theorem~\ref{thm:well-behaviour-of-composition} which says that the $\Cst$\nb-correspondence associated with a composite is isomorphic to the composite of the $\Cst$\nb-correspondences.

The last section contains examples. Most of the examples are related to the ones in~\cite{Holkar2015Construction-of-Corr}.

\newpage
\section{Preliminaries}\label{prel}
\subsection{Revision}
\label{rev}
The symbols $\iso,\homeo,\R^+$ and $\R^+_*$ stand for isomorphic, homeomorphic, the set of positive real numbers and the multiplicative group of positive real numbers, respectively. The symbol $\otimes$ and $\hat\otimes$ indicate the algebraic tensor product modules and the interior tensor product Hilbert modules, respectively.
 
We work with continuous families of measures and all the measures are assumed to be positive, Radon and $\sigma$\nb-finite. The families of measures are denoted by small Greek letters and the corresponding integration function that appears in the continuity condition is denoted by the Greek upper case letter used to denote the family of measures. For example, if $\lambda$ is a family of measures along a map $f\colon X\to Y$, then $\Lambda\colon \Contc(X)\to \Contc(Y)$ is the function $\Lambda(f)(y)=\int_Xf\,\dd^y$. The capitalisations of $\alpha$, $\beta$, $\chi$ and $\mu$ are  $A$, $B$, $\chi$ and $M$, respectively. 

However, for a single measure on a space, which is a family of measures along the constant map onto a point, we follow the traditional convention, that is, the same letter is used to denote the measure and the corresponding integration functional. For example, if $\alpha$ is a measure on $X$, then $\alpha(f)=\int_X f\,\dd\alpha$ for $f\in \Contc(X)$.

Let $G$ be a groupoid, then $r_G,s_G$ and $\textup{inv}_G$ denote the source, range and the inversion maps for $G$. Given a left $G$\nb-space $X$, we tacitly assume that the momentum map for the action is $r_X$. If $X$ is a right $G$\nb-space, then $s_X$ is the momentum map for the action.

 We denote $G\times_{s_G,\base[G],r_X}X$, the fibre product for $G$ and $X$ over $\base[G]$ along $s_G$ and $r_X$, by $G\times_{\base[G]}X$. If $X$ is a right $G$\nb-space, then $X\times_{\base[G]}G$ has a similar meaning.  For a left $G$\nb-space $X$, $G\ltimes X$ is the transformation groupoid and its set of arrows is the fibre product $G\times_{\base[G]}X$. Similar is the meaning of $X\rtimes G $ a right $G$\nb-space $X$.

For $A,B\subseteq \base[G]$ we define $G^A=r_G\inverse(A), G_B=s_G\inverse(B)$ and $G^A_B=G^A\cap G_B$. When $A=\{u\}$ and $B=\{v\}$ are singletons, we simply write $G^u$, $G_v$ and $G^u_v$ instead of $G^{\{u\}}$, $G_{\{v\}}$ and $G^{\{u\}}_{\{v\}}$, respectively. Let $X$ and $Y$ be left and right $G$\nb-spaces, respectively, and let $A\subseteq \base[G]$ and $u\in \base[G]$. Then $X^A, X^u,Y_B$ and $Y_u$ have the similar obvious meanings.

We denote a $\Cst$\nb-correspondence only by the Hilbert module involved in it; we do not write the representation of the left $\Cst$\nb-algebra. Thus we say `$\Hils$ is a $\Cst$\nb-correspondence from a $\Cst$\nb-algebra $A$ to $B$', and not `$(\Hils,\phi)$ is a $\Cst$\nb-correspondence from a $\Cst$\nb-algebra $A$ to $B$' where $\phi\colon A\to\Bound_B(\Hils)$ is the nondegenerate {\Star}representation involved in the definition of the correspondence. We also write `$\Hils\colon A \to B$ is a $\Cst$\nb-correspondence'.

\label{page:compo-of-hilm}Now we sketch the process of composing the $\Cst$\nb-correspondences briefly and explain a few notation along the way. 
%
%
Let $A,B$ and $C$ be $\Cst$\nb-algebras, and let $\Hils\colon A\to B$ and $\Hils[F]\colon B\to C$ be $\Cst$\nb-correspondences. 
Endow $\Hils\otimes_{\C}\Hils[F]$ with the inner product $\inpro{\zeta\otimes\xi}{\zeta'\otimes\xi'}=\inpro{\xi}{\inpro{\zeta}{\zeta'}\xi'}$. Let $N\subseteq\Hils\otimes_{\C}\Hils[F]$ be the closed vector subspace of the vectors of zero norm, that is, $N=\{z\in\Hils\otimes_{\C}\Hils[F]: \inpro{z}{z}=0\}$. 
The proof of Proposition~{4.5} in~\cite{Lance1995Hilbert-modules} shows that the subspace $N$ is same as the the subspace spanned by the elements of the form $\zeta b\otimes\xi-\zeta\otimes b\xi$ where $\zeta\in \Hils, \xi\in\Hils[F]$ and $b\in B$.
The Hilbert $C$\nb-module $\Hils\hat\otimes_B\Hils[F]$ is the completion of $(\Hils\otimes_{\C}\Hils[F])/N$ in the norm induced by $\inpro{}{}$. We denote the equivalence class of $\zeta\otimes\xi\in\Hils\otimes_{\C}\Hils[F]$ in $\Hils\hat\otimes_B\Hils[F]$ by $\zeta\hat\otimes\xi$. The action of $A$ on $\Hils\hat\otimes_B\Hils[F]$ is $a(\zeta\hat\otimes \xi)=a\zeta\hat\otimes\xi$ where $a\in A$ and $\zeta\hat\otimes\xi\in\Hils\hat\otimes\Hils[F]$.
We call the map $\xi\otimes_{\C}\zeta\mapsto \xi\hat\otimes\zeta$,  $\Hils\otimes_{\C}\Hils[F] \to \Hils\hat\otimes\Hils[F]$, the obvious map of Hilbert $C$\nb-modules which, clearly, has a dense image.

\begin{definition}[Topological correspondence]
  \label{def:correspondence}
A \emph{topological correspondence} from a locally compact groupoid $G$ with a Haar
system $\alpha$ to a locally compact groupoid $H$ equipped with a Haar system $\beta$ is a pair $(X, \lambda)$, where:
	\begin{enumerate}[label=\textit{\roman*}), leftmargin=*]
		\item $X$ is a locally compact $G$-$H$-bispace,
                \item the action of $H$ is proper,
              	\item $\lambda = \{\lambda_u\}_{u\in\base}$ is an
                  $H$\nb-invariant proper continuous family of measures along the momentum map $s_X\colon X\to\base$,
		\item  there exists a continuous function $\Delta: G
                  \ltimes X \rightarrow \R^+$ such that for each $u
                  \in \base$ and $F\in \Contc(G\times_{\base[G]} X)$,
                  \begin{multline*}
                    \int_{X_{u}} \int_{G^{r_X(x)}} F(\gamma\inverse,
                    x)\, \dd\alpha^{r_X(x)}(\gamma)\, \dd\lambda_{u}(x) \\=
                    \int_{X_{u}} \int_{G^{r_X(x)}} F(\gamma,
                    \gamma\inverse x)\, \Delta(\gamma, \gamma\inverse
                    x) \, \dd\alpha^{r_X(x)}(\gamma)
                    \,\dd\lambda_{u}(x). 
                \end{multline*}
	\end{enumerate}
\end{definition}

The function $\Delta$ is unique and is called \emph{the adjoining function} of the
correspondence. 

For $\phi \in \Contc(G)$, $f \in \Contc(X)$ and $\psi \in \Contc(H)$ define
the functions $\phi\cdot f$ and $f\cdot \psi$ on $X$ as follows:
\begin{equation}\label{def:left-right-action}
 \left\{\begin{aligned}
(\phi\cdot f)(x) &\defeq \int_{G^{r_X(x)}} \phi(\gamma) f(\gamma\inverse x) \,
\Delta^{1/2}(\gamma, \gamma\inverse x) \; \dd \alpha^{r_X(x)}(\gamma),\\
(f\cdot\psi )(x) &\defeq \int_{H^{s_X(x)}} f(x\eta)\psi({\eta}\inverse) \; \dd
\beta^{s_X(x)}(\eta).
  \end{aligned}\right.
\end{equation}
For $f, g \in \Contc(X)$ define the function $\langle f, g \rangle $ on $H$ by
\begin{align}\label{def:inner-product}
\langle f, g \rangle (\eta) &\defeq \int_{X_{r_H(\eta)}} \overline{f(x)} g(x\eta) \; \dd
\lambda_{r_H(\eta)}(x).
\end{align}
Very often we write $\phi f$ and $f\psi$ instead of $\phi\cdot f$
and $f\cdot \psi$, respectively.~\cite{Holkar2015Construction-of-Corr}*{Lemma 2.10} proves that $\phi f, f\psi\in \Contc(X)$ and $\inpro{f}{g}\in \Contc(H)$.
\begin{theorem}[\cite{Holkar2015Construction-of-Corr}*{Theorem 2.39}]\label{thm:mcorr-gives-ccorr}
Let $(G, \alpha)$ and $(H, \beta)$ be locally compact groupoids with Haar systems. Then a topological correspondence $(X, \lambda)$ from $(G, \alpha)$ to $(H, \beta)$ produces a\; $\Cst$\nb-correspondence $\Hils(X)$ from $\Cst(G,\alpha)$ to $\Cst(H,\beta)$.
\end{theorem}
\subsection{Cohomology of proper groupoids}
In this subsection, we show that the first continuous cohomology group with real coefficients is trivial. It can be readily checked that the result is valid for the groupoid equivariant continuous cohomology introduced in~\cite{Holkar2015Construction-of-Corr}*{Section 1}, and also for the (equivariant) Borel cohomology of a proper (topological) groupoid.
\begin{lemma}[Lemma 1, Appendix I in~\cite{Bourbaki2004Integration-II-EN}]\label{lemma:cutoff-function-for-equ-relation}
  Let $X$ be a locally compact Hausdorff space, $R$ an open equivalence relation in $X$, such that the quotient space $X/R$ is paracompact; let $\pi$ be the canonical mapping of $X$ onto $X/R$. There is a continuous real-valued function $F\geq0$ on $X$ such that:
  \begin{enumerate}[label=\roman*)]
  \item $F$ is not identically zero on any equivalence class with respect to $R$;
  \item for every compact subset $K$ of $X/R$, the intersection of $\pi\inverse(K)$ with $\supp(F)$ is compact.
  \end{enumerate}
\end{lemma}

A continuous map $f\colon X\to Y$ is proper, if for each $y\in Y$, $f\inverse(y)\subseteq X$ is quasi-compact. We call a locally compact groupoid proper if the map $(r_G, s_G)\colon G\to\base[G]\times\base[G]$, $(r_G, s_G)(\gamma)=(r_G(\gamma),s_G(\gamma))$, is proper.
\begin{proposition}[\cite{Tu2004NonHausdorff-gpd-proper-actions-and-K}*{Proposition 2.10}]
 \label{prop:prop-gpd-equi}
Let $G$ be a locally compact groupoid. Then the following assertions are equivalent.
\begin{enumerate}[label=\roman*)]
\item $G$ is proper;
\item the map $(r_G, s_G)\colon G\to\base[G]\times\base[G]$ is closed and for each $u\in\base[G]$, $G^u_u\subseteq G$ is quasi-compact;
\item for all quasi-compact subsets $K, L\subseteq\base[G]$, $G^L_K$ is quasi-compact;
\item for all compact subsets $K, L\subseteq\base[G]$, $G^L_K$ is compact;
\item for every quasi-compact subsets $K\subseteq\base[G]$, $G^K_K$ is quasi-compact;
\item for all $x,y\in\base[G]$, there are compact neighbourhoods $K_x$ and $L_y$ of $x$ and $y$, respectively,such that $G^{K_x}_{L_y}$ is quasi-compact.
\end{enumerate}
\end{proposition}
\cite{Tu2004NonHausdorff-gpd-proper-actions-and-K}*{Proposition 2.10} is stated for groupoids which do not, necessarily, have Hausdorff space of units in which case (i)-(v) are equivalent and (v)$\implies$(vi).
\begin{lemma}\label{lemma:proper-gpd-has-prob-measures}
  Let $(G,\alpha)$ is a locally compact proper groupoid with a Haar
  system. If $G\7\base[G]$ is paracompact then there is a continuous invariant
  family of probability measures on $G$. Furthermore, each measure in this family has a compact support.
\end{lemma}
\begin{proof}
  Since $G$ has a Haar system, the range map of $G$ is open, hence
  the quotient map $\pi\colon \base[G]\to G\backslash\base[G]$ is open. Since $G$ is
  proper, $G\backslash\base[G]$ is locally compact and Hausdorff. By hypothesis $G\7\base[G]$ is paracompact. Now we apply Lemma~\ref{lemma:cutoff-function-for-equ-relation} to get a function $F$ on $\base[G]$ such that $F$ is not identically zero on any $G$-orbit in $\base[G]$ and for every compact $K\subseteq G\backslash\base[G]$ the intersection
  $\supp(F)\cap\pi\inverse(K)$ is compact. Define $h\colon \base[G]\to
  \R^+$ by
\[
h(u)=\int_{G^u} F\circ s_G(\gamma) \;\dd\alpha^u(\gamma).
\]
Property (ii) of $F$ from Lemma~\ref{lemma:cutoff-function-for-equ-relation} and the full
support condition of $\alpha^u$ imply that $h(u)>0$. To see that
$h(u)<\infty$, notice that $\supp(F\circ s_G)\cap G^u\subseteq G$ is
compact:
\[
\gamma\in \supp(F\circ s_G)\cap G^u \implies s_G(\gamma)\in\supp(F)
\textnormal{ and } r_G(\gamma)=u.
\]
Thus if $\tilde u\subseteq \base[G]$ denotes the orbit of $u$, then
$\supp(F\circ s_G)\cap G^u\subseteq (r_G, s_G)\inverse(\{u\}\times\supp(F|_{\tilde u}))$. Property (ii) of $F$ from Lemma~\ref{lemma:cutoff-function-for-equ-relation} says that $\supp(F|_{\tilde u})$ is compact. As $G$ is a proper groupoid, the set $(r_G\times s_G)\inverse(\{u\}\times\supp(F|_{\tilde u}))$ is compact which implies that $\supp(F\circ s_G)\cap G^u$ is compact.

Using the invariance of $\alpha$, it is not hard to see that the function $h$ is constant on the orbits of $\base[G]$. Put
$F'=(F/h)\circ s_G$, then
\begin{equation}
  \label{eq:cut-off-function}
\int_{G^u} F'(\gamma) \;\dd\alpha^u(\gamma) =1.
\end{equation}
Denote $F'\,\alpha^u$ by $p^u$, then $p\defeq \{p^u \}_{u\in\base[G]}$ is a family of probability measures on $G$. Explicitly, $p$ is given by 
\[
\int_{G^u} f\,\dd p^u=\int_{G^u}f(\gamma)\;F'(\gamma)\;\dd\alpha^u(\gamma)
\]
for $f\in \Contc(G)$. It follows from the definition of measure \(p^u\), where \(u\in\base[G]\), that the the compact set $\supp(F\circ s_G)\cap G^u$ is the support of \(p_u\).

To check that $p$ is invariant, let $f\in \Contc(G)$ and $\eta\in G$, by the definition of $p$ we have
\[
\int_{G^{s_G(\eta)}} f(\eta\gamma)\;\dd p^{s_G(\eta)}(\gamma)=
\int_{G^{s_G(\eta)}} f(\eta\gamma)\; F'(\gamma) \;\dd\alpha^{s_G(\eta)}(\gamma).
\]
Now change the variable $\eta\gamma\mapsto \gamma$ to that the previous term equals 
\[
\int_{G^{r_G(\eta)}} f(\gamma)\; F'(\eta\inverse \gamma) \;\dd\alpha^{r_G(\eta)}(\gamma).
\]
Use the invariance of $\alpha$ and the fact that $F'(\eta\inverse\gamma)\defeq F\circ s_G(\eta\inverse\gamma)/h\circ s_G(\eta\inverse\gamma)=F\circ s_G(\gamma)/h\circ s_G(\gamma)=F'(\gamma)$ and compute further:
\begin{align*}
\int_{G^{r_G(\eta)}} f(\gamma)\; F'(\eta\inverse \gamma) \;\dd\alpha^{r_G(\eta)}(\gamma)
&=\int_{G^{s_G(\eta)}} f(\gamma)\; F'(\gamma) \;\dd\alpha^{r_G(\eta)}(\gamma)\\
&=\int_{G^{s_G(\eta)}} f(\gamma)\;\dd p^{r_G(\eta)}(\gamma).\qedhere
\end{align*}
\end{proof}
\begin{remark}
  \label{rem:amenabiltiy-prob-measure}
Anantharaman Delaroche and Renault define \emph{topological} amenability for a locally compact topological groupoid (\cite{Anantharaman-Renault2000Amenable-gpd}*{Definition 2.2.7}). Lemma~\ref{lemma:proper-gpd-has-prob-measures} says that every proper groupoid $G$ with a Haar system and $G\7\base[G]$ paracompact is topologically amenable.
\end{remark}

\begin{proposition}\label{prop:proper_gpd_first_cohom_zero}
Let $G$ be a locally compact proper groupoid and $\alpha$ a Haar system on
$G$. Then every $\R$\nb-valued 1-cocycle is
a coboundary, that is, $\HH^1(G; \R) = 0$.
\end{proposition}
\begin{proof}
Let $p=\{p^u\}_{u\in \base[G]}$ be an invariant family of
probability measures on $G$ which is obtained using Lemma~\ref{lemma:proper-gpd-has-prob-measures}.
We claim that for a 1-cocycle $c\colon G \rightarrow \R$, the function 
\[\underline{b}(u) = \int_G c(\gamma) \,\dd p^{u}(\gamma) \text{ for } u\in\base[G]\]
satisfies $c=\underline{b} \circ s - \underline{b}\circ r$. Lemma~\ref{lemma:proper-gpd-has-prob-measures} says that the support of each measure in \(p\) is compact, hence the above integral is well-defined. To see that \(\underline{b}\) is the desired cochain, let $\eta\in G $ and compute: 
  \begin{align*}
   (\underline{b} \circ s- \underline{b}\circ r)(\eta) 
&= \int_{G^{s_G(\eta)}} c(\gamma) \,\dd p^{s(\eta)}(\gamma)-\int_{G^{r_G(\eta)}} c(\gamma) \,\dd p^{r(\eta)}(\gamma)\\
&= \int_{G^{s_G(\eta)}} c(\gamma) \,\dd p^{s(\eta)}(\gamma)-\int_{G^{s_G(\eta)}} c(\eta\gamma) \,\dd p^{s(\eta)}(\gamma)\\
&= \int_{G^{s_G(\eta)}} \left(c(\gamma) - c(\eta)+c(\gamma)\right)\dd p^{s(\eta)}(\gamma)\\
&= c(\eta) \int \dd p^{s(\eta)}(\gamma) = c(\eta).
  \end{align*}
We used the invariance of $p$ to get the second equality above.
\end{proof}

\section{Composition of correspondences}
\subsection{Preparation for composition}
\label{sec:composition-of-corr}
Let $G$ be a locally compact proper groupoid with $G/\base[G]$ paracompact. Since $G$ is proper the action of $G$ on $\base[G]$ given by $\gamma\cdot s_G(\gamma)=r_G(\gamma)$ for $\gamma\in G$ is proper. Let $\lambda$ be a Haar system on $G$. Then $\lambda$ induces a family of measures $[\lambda]$ along $\pi\colon \base[G]\to G/\base[G]$.  And $[\lambda]$ is defined as
\begin{equation}
  \label{eq:def-of-quo-lambda}
\int_{\base[G]} f\,\dd[\lambda]^{[u]}=\int_{G} f(\gamma\inverse\cdot u))\,\dd\lambda^u(\lambda)=\int_G f\circ s_G(\gamma)\,\dd\lambda^u(\gamma)
\end{equation}
for $f\in \Contc(\base[G])$. Note that in Equation~\ref{eq:def-of-quo-lambda}, $\gamma\inverse\cdot u$ does not stand for the composite of $\gamma\inverse$ and $u$ but for the action of $G$ on $\base[G]$, that is, $\gamma\inverse\cdot u=r_G(\gamma\inverse)=s_G(\gamma)$. We draw Figure~\ref{fig:symmetric-measures-1} which contains all this data.
\begin{figure}[htb]
  \centering
\[\begin{tikzcd}[scale=2]
G \arrow{r}{\lambda\inverse}[swap]{s_G} \arrow{d}{r_G}[swap]{\lambda}
 & \base[G] \arrow{d}{[\lambda]}[swap]{\pi} \\
 \base[G] \arrow{r}{[\lambda]}[swap]{\pi} 
&G/\base[G]
 \end{tikzcd}\]
\caption{}
\label{fig:symmetric-measures-1}
\end{figure}
 
A measure $m$ on $\base[G]$ induces measures $m\circ\lambda$ and $m\circ\lambda\inverse$ on $G$. For $f\in \Contc(G)$ 
\[
\int_G f\,\dd(m\circ\lambda\inverse)\defeq \int_{\base[G]}\int_{G} f(\gamma\inverse)\,\dd\lambda^u(\gamma)\,\dd m(u).
\]
The measure $m\circ\lambda$ is defined similarly. We call the measure $m$ on \(\base[G]\) \emph{invariant} with respect to \((G,\lambda)\) if $m\circ\lambda= m\circ\lambda\inverse$ and the measure \(m\circ\lambda\) on \(G\) is called \emph{symmetric}.
\begin{proposition}\label{prop:pushing_measure}
Let $G$ be a proper groupoid with $G/\base[G]$ paracompact, let $\lambda$ be a Haar system for $G$ and $\pi\colon \base[G] \rightarrow G/\base[G]$ the quotient map.
\begin{enumerate}[label= \emph{\roman*)}]
	\item Let $\mu$ be a measure on $G/\base[G]$ and let $m$ denote the measure $\mu\circ[\lambda]$ on $\base[G]$. Then $m$ is an invariant measure.
	\item Let $m$ be a measure on $\base[G]$. If $m$ is invariant, then there is a measure $\mu$ on $G/\base[G]$ with $\mu\circ [\lambda] = m$.
        \item The measure $\mu$ in \textup{(}ii\textup{)}, with $\mu\circ[\lambda]=m$, is unique.
\end{enumerate}
\end{proposition}
\begin{proof}
(i): Let $f\in \Contc(G)$, then 
\begin{align*}
\int_G f\,\dd(m\circ\lambda)&=\int_G f\,\dd(\mu\circ[\lambda]\circ\lambda)=\int_{G/\base[G]}\int_{\base[G]} \Lambda(f)(u)\;\dd[\lambda]^{[u]}(u)\,\dd\mu([u])\\
&=\int_{G/\base[G]}\int_{G^u} \Lambda(f)\circ s_G(\gamma)\;\dd\lambda^{u}(\gamma)\,\dd\mu([u])\\
&=\int_{G/\base[G]}\int_{G^u}\left(\int_{G^{s_G(\gamma)}} f(\eta)\,\dd\lambda^{s_G(\gamma)}(\eta)\right)\dd\lambda^{u}(\gamma)\,\dd\mu[u].
\end{align*}
We know that $s_G(\gamma)=r_G(\gamma\inverse)$. Now change the variable $\eta\mapsto\gamma\inverse\eta$, and use the left invariance of $\lambda$ to see that the previous term equals
\begin{align*}
\int_{G/\base[G]}\int_{G}\int_G f(\gamma\inverse\eta)\,\dd\lambda^{r_G(\gamma)}(\eta)\dd\lambda^u(\gamma)\,\dd\mu([u]).
\end{align*}
We have removed the superscripts of $G$ in above equation for simplicity. Now apply Fubini's theorem to $\dd\lambda^{r_G(\gamma)}(\eta)\,\dd\lambda^u(\gamma)$ which is allowed since $r_G(\gamma)=u$ and $f$ is compactly supported continuous function. Moreover, note that $u=r_G(\gamma)=s_G(\gamma\inverse)$, use the right invariance of \(\lambda\inverse\) and compute further:
\begin{align*}
&\int_{G/\base[G]}\int_{G}\int_G f(\gamma\inverse\eta)\,\dd\lambda^{s_G(\gamma\inverse)}(\gamma)\,\dd\lambda^u(\eta)\,\dd\mu([u])\\
&=\int_{G/\base[G]}\int_{G}\int_G f(\gamma\inverse)\,\dd\lambda^{s_G(\eta)}(\gamma)\,\dd\lambda^u(\eta)\,\dd\mu([u]) \\
&=\int_{G/\base[G]}\int_{G} \Lambda\inverse(f)\circ s_G(\eta)\,\dd\lambda^u(\eta)\,\dd\mu([u])\\
&=\int_{G/\base[G]} f\,\dd(\mu\circ[\lambda]\circ\lambda\inverse)=\int_G f\,\dd(m\circ\lambda\inverse).
\end{align*}
To be precise, in the last the first equality is obtained by using the right invariance of \(\lambda\inverse\).

\noindent (ii): Let $e$ be a function on $\base[G]$ which is similar to function $F/h$ in Lemma~\ref{lemma:proper-gpd-has-prob-measures}, and thus $\Lambda(e\circ s_G)=1$. Function $e\circ s_G$ is like function $F'$ in Equation~\ref{eq:cut-off-function}. Let $\mu$ be the measure on $G/\base[G]$ which is defined as $\mu(g) = m((g\circ\pi)\cdot e )$ for $g \in \Contc(G/\base[G])$. Let $[\Lambda]$ denote the integration function corresponding to $[\lambda]$. For $f\in \Contc(\base[G])$
\begin{align*}
\int_{\base[G]} f\,\dd(\mu\circ[\lambda])&= \int_{G/\base[G]} [\Lambda](f)([u])\,\dd\mu([u])= \int_{\base[G]} [\Lambda](f)\circ\pi(u) \,e(u)\,\dd m(u)\\
&= \int_{\base[G]}\int_{G} f\circ s_G(\gamma)\, e(r_G(\gamma))\;\dd\lambda^{u}(\gamma)\,\dd m(u).
\end{align*}
We change $\gamma\mapsto \gamma\inverse$, then use the symmetry of the measure $m\circ\lambda$ and continue the computation:
\[
 \int_{\base[G]}\int_{G} f\circ r_G(\gamma)\, e(s_G(\gamma))\;\dd\lambda^{u}(\gamma)\,\dd m(u)=\int_{\base[G]} f(u)\,\Lambda(e\circ s_G)(\gamma)\,\dd m(u)
=\int f\,\dd m.
\]
The last equality is due to the property of $e$ that $\Lambda(e\circ s_G)=1$.

\noindent (iii): Let $\mu'$ be another measure on $G/\base[G]$ which satisfies the
 condition $\mu'\circ[\lambda]=m$. Since the integration map $[\Lambda]\colon \Contc(\base[G])\to \Contc(G/\base[G])$ is surjective, $\mu\circ[\lambda] = \mu'\circ[\lambda]$ implies $\mu=\mu'$.
\end{proof}

Now we study the case when the measure $m$ is not invariant, but \emph{strongly quasi-invariant}; $m$ is called \emph{quasi-invariant} with respect to \((G,\lambda)\) if $m \circ \lambda \sim m \circ\lambda\inverse$. Following Folland, see~\cite{Folland1995Harmonic-analysis-book}*{Chapter 2, Section 6, page 58}, we call \(m\) \emph{strongly} quasi-invariant with respect to \((G,\lambda)\) if there is a continuous homomorphism $\Delta: G \to \R^+_*$ with $m \circ \lambda = \Delta \cdot (m \circ\lambda\inverse)$, that is, \(m\) is quasi-invariant with respect to \((G,\lambda)\) and the Radon-Nikodym derivative implementing the equivalence of the measures \(m \circ\lambda\) and \(m \circ\lambda\inverse\) is continuous. Very often, when there is no chance of confusion, we drop the phrase `with respect to \((G,\lambda)\)' while talking about a (strongly) quasi-invariant measure. The cohomology theory of groupoids tells us that a homomorphism from $G$ to an abelian group $R$ is same as an $R$\nb-valued 1\nb-cocycle (See~\cite{Holkar2015Construction-of-Corr}*{Section 1}).

Let $(G,\lambda)$ be as in Proposition~\ref{prop:pushing_measure}. Let $m$ be a strongly quasi-invariant measure on $\base[G]$, and let $\Delta$ be the $\R^+_*$\nb-valued continuous $1$\nb-cocycle which implements the quasi-invariance. Then $\Delta$ gives an $\R$\nb-valued 1\nb-cocycle $\log \circ \Delta: G \rightarrow \R$. Proposition~\ref{prop:proper_gpd_first_cohom_zero} says that $\log \circ\, \Delta = \underline{b} \circ s_G - \underline{b}\circ r_G$ for some continuous function $\underline{b}: \base[G] \rightarrow \R$. Thus
\[
\Delta = \frac{ \e^{\underline{b} \circ s_G}}{\e^{\underline{b} \circ r_G}}  .
\]
Write $b = \e^{\underline{b}}$, then $b>0$ and it can be checked that
\[
\Delta
= \frac{ \e^{\underline{b} \circ s_G}}{\e^{\underline{b} \circ r_G}}
=\frac{ \e^{\underline{b}} \circ s_G}{\e^{\underline{b}}\circ r_G}
= \frac{ b \circ s_G}{b\circ r_G}.
\]
Rewriting the definition of \((G,\lambda)\)\nb-quasi-invariance of $m$ using the above value of $\Delta$ gives that $m \circ \lambda = \left(\frac{ b \circ s_G}{b\circ r_G} \right) m \circ \lambda\inverse$ which is equivalent to $(b\circ r_G) (m \circ \lambda) = (b \circ s_G)(m \circ\lambda\inverse)$. A straightforward calculation shows that $(b\circ r_G) (m \circ \lambda) = (bm)\circ \lambda$ and $(b\circ s_G)(m \circ \lambda\inverse) = (bm)\circ \lambda\inverse$. 
Thus we get
\begin{proposition}\label{prop:weak_equ_implies_equ}
Let $(G,\lambda)$ be a locally compact proper groupoid with a Haar system. Assume that $G/\base[G]$ is paracompact. Let $m$ be a strongly \((G,\lambda\))\nb-quasi-invariant measure on $\base[G]$. Let $\Delta$ be the $\R^+_*$\nb-valued continuous $1$\nb-cocycle which implements the quasi-invariance. Then there is a continuous function $b: \base[G] \to \R^+$ with 
\begin{enumerate}[label=\roman*)]
\item $\frac{b\circ s_G(\gamma)} {b\circ r_G(\gamma)} = \Delta(\gamma)$ for all $\gamma\in G$;
\item  the measure $bm$ on $\base[G]$ is \((G,\lambda)\)\nb-invariant, that is, $bm\circ\lambda=bm\circ\lambda\inverse$.
\end{enumerate}
\end{proposition}

\subsection{Composition of topological correspondences}
\label{subs:composition}
Let $(X, \alpha)$ and $(Y, \beta)$ be
correspondences from $(G_1,\chi_1)$ to
$(G_2,\chi_2)$ and from $(G_2,\chi_2)$ to $(G_3,\chi_3)$,
respectively. Let $\Delta_1$ and $\Delta_2$ be the adjoining functions of $(X, \alpha)$ and $(Y, \beta)$, respectively. Additionally, assume that $X$ and $Y$ are Hausdorff and second countable. We draw Figure~\ref{fig:composition-details} that comprises of this data.
\begin{figure}[htb]
  \centering
\[
\begin{tikzcd}[column sep=small]
 & X \arrow[dashrightarrow]{dl}[description]{\Delta_1}\arrow[dashrightarrow]{dr}{\alpha}&  &Y \arrow[dashrightarrow]{dl}[description]{\Delta_2}\arrow[dashrightarrow]{dr}{\beta}\\
  (G_1,\chi_1) \arrow{rr} &                        & (G_2,\chi_2)\arrow{rr}  &                & (G_3,\chi_3)
\end{tikzcd}
\]  
\caption{}
\label{fig:composition-details}
\end{figure}

We need to create a $G_1$-$G_3$\nb-bispace $\Omega$ equipped with a $G_3$\nb-invariant continuous family of measures $\mu=\{\mu_u\}_{u\in\base_3}$ with each $\mu_u$ $G_1$\nb-quasi-invariant. And the $\Cst(G_1,\chi_1)$\nb-$\Cst(G_3,\chi_3)$\nb-Hilbert module $\Hils(\Omega)$ should be isomorphic to the Hilbert module $\Hils(X) \hat\otimes_{\Cst(G_2,\chi_2) }\Hils(Y)$. 

Denote the fibred product $X\times_{\base[G_2]}Y$ by $Z$. Then $Z$ carries the diagonal action of $G_2$. Since the  action of $G_2$ on $X$ is proper, its action on $Z$ is proper. Thus the the transformation groupoid $Z\rtimes G_2$ is proper. We define the space $\Omega=Z/ G_2=\base[(Z\rtimes G_2)]/(Z\rtimes G_2)$. Since $Z$ is locally compact, Hausdorff and second countable, so is $\Omega$. Being a locally compact, Hausdorff and second countable space, $\Omega$ is paracompact.

The following discussion in this section goes through under a milder hypothesis, namely, $X$ and $Y$ are locally compact Hausdorff and $\Omega$ is paracompact.
\begin{observation}
  \label{obs:omega-is-G1-G3-bispace}
The space $Z$ is a $G_1$-$G_3$-bispace. The momentum maps are
$r_{Z}(x,y)=r_X(x)$ and $s_{Z}(x,y)=s_Y(y)$. For $(\gamma_1, (x,y))\in G_1\times_{\base[G_1]}Z$ and $((x,y),\gamma_3)\in Z\times_{\base[G]_3} G_3$, the actions are $\gamma_1\cdot (x,y)=(\gamma_1x,y)$ and $(x,y)\cdot\gamma_3=(x,y\gamma_3)$, respectively. These actions descend to $\Omega$ and make it a $G_1$\nb-$G_3$\nb-bispace. Thus $r_\Omega([x,y])=r_X(x)$ and $s_\Omega([x,y])=s_Y(y)$ and $\gamma_1[x,y]\gamma_2=[\gamma_1x,y\gamma_2]$ for appropriate $\gamma_1\in G_1, [x,y]\in \Omega $ and $\gamma_2$.
\end{observation}
\begin{lemma}
  \label{lemma:composite-is-right-proper}
The right action of $G_3$ on $\Omega$ defined in Observation~\ref{obs:omega-is-G1-G3-bispace} is proper.
\end{lemma}
\begin{proof}
Follows from~\cite{Tu2004NonHausdorff-gpd-proper-actions-and-K}*{Lemma 2.33}.
\end{proof}

 For each $u \in {\base[G]}_3$ define a measure $m_u$ on the space $Z$ as follows: for $f\in \Contc(Z)$
\[
\int_{Z} f \;\dd m_u = \int_Y\int_X f(x,y) \; \dd\alpha_{r_Y(y)}(x)\; \dd\beta_{u}(y).
\]
\begin{lemma}\label{lemma:m_G_3_family}
The  family of measures $\{m_u \}_{u \in {\base[G]}_3}$ is a 
$G_3$\nb-invariant continuous family of measures on $Z$.
\end{lemma}
\begin{proof}
It is a routine computation to check that the $G_3$\nb-invariance of the family of measures $\beta$ makes $\{m_u\}_{u\in{\base[G]}_3}$ $G_3$\nb-invariant. The computation is similar to that in Proposition~\ref{prop:mu_u_is_measure}. To check the continuity let $f\in \Contc(X)$ and $g\in \Contc(Y)$, then 
\[
\int_Z f\otimes g\,\dd m_u = B((A(f)\circ r_Y)\;g)(u).
\]
which is in $\Contc(\base[G_3])$. Now use the theorem of Stone-Weierstra{\ss} to see that the set $\{f\otimes g: f\in \Contc(X), g\in \Contc(Y)\}\subseteq \Contc(Z)$ is dense which concludes the lemma.
\end{proof}

 The Haar system $\chi_2$ of $G_2$ induces a Haar system $\chi$ on $Z\rtimes G_2$; for $f\in \Contc(Z\rtimes G_2)$ and $(x,y)\in Z$
\[
\int_{(Z\rtimes G_2)^{(x,y)}} f\,\dd\chi^{(x,y)}\defeq\int_{G^{s_X(x)}_2} f((x,y),\gamma)\,\dd\chi_2^{s_X(x)}(\gamma).
\]
The quotient map $\pi : Z \rightarrow \Omega$ carries the family of measures $[\chi]$; the definition of which is similar to the one in Equation~\ref{eq:def-of-quo-lambda}. We write $\lambda$ instead of $[\chi]$, and $\lambda^\omega$ instead of $[\chi]^\omega$ for all ${\omega\in\Omega}$. Recall from Subsection~\ref{sec:composition-of-corr} that for $f \in \Contc(Z)$ and $\omega=[x,y]\in\Omega$
\begin{equation*}
  \label{eq:middle-measure-family}
\int_{\pi\inverse(\omega)} f\; \dd\lambda^{\omega} \defeq \int_{G_2^{r_Y(y)}} f(x\gamma, \gamma\inverse y) \; \dd\chi_2^{r_Y(y)}(\gamma).
\end{equation*}

We wish to prove that, up to equivalence, $\{ m_u\}_{u\in {\base[G_3]}}$ can be pushed down from $Z$ to $\Omega$ to a $G_3$\nb-invariant family of measures $\{ \mu^u\}_{u\in {\base[G_3]}}$. We use $\lambda$ to achieve this. To be precise, we find a continuous function $b\colon Z\to\R^+$ and a family of measures $\mu$ on $\Omega$ which gives a disintegration $bm=\mu\circ\lambda$. Before we proceed we prove a small lemma.
\begin{lemma}\label{lemma:existance_of_composed_measure}
Let $(X, \alpha)$ and $(Y, \beta)$ be correspondences from $(G_1,\chi_1)$ to $(G_2,\chi_2)$ and from $(G_2,\chi_2)$ to $(G_3,\chi_3)$, respectively, with $\Delta_1$ and $\Delta_2$ as their adjoining functions. Then for each $u \in {\base[G]}_3$ there is a function $b_u$ on $Z$ such that the measure $b_u m_u$ on  on $Z=\base[(Z\rtimes G_2)]$ is an invariant measure with respect to \((Z\rtimes G_2,\chi)\). Furthermore, $b_u$ satisfies the relation $b(x\gamma,\gamma\inverse y)b(x,y)\inverse=\Delta_2(\gamma\inverse, y)$.
\end{lemma}
 We work with a single $\mu_u$ at a time, so we prefer to drop the suffix $u$ of $b_u$ and simply write $b$. Using Lemma~\ref{lemma:existance_of_composed_measure} we define the function $\Delta\colon Z\times_{\base[G_2]}G_2\to\R^+$ as $\Delta((x,y),\gamma)=\Delta_2(\gamma\inverse, y)$.
\begin{proof}
The proof follows the steps below:
\begin{enumerate}[label=\roman*)]
\item Firstly, we show that for each $u\in{\base[G]}_3$, $m_u$ is strongly quasi-invariant with respect to \((Z\rtimes G_2,\chi)\) and $\Delta((x,y),\gamma)\defeq\Delta_2(\gamma\inverse, y)$ is the cocycle which implements the quasi-invariance.
\item Since $ Z\rtimes G_2$ is proper, we appeal to Proposition~\ref{prop:weak_equ_implies_equ} to get a function $b\colon Z=\base[(Z\rtimes G_2)]\to \R^+$ having the desired properties.
\end{enumerate}
\noindent (i):  We draw Figure~\ref{fig:pushig-measure-down} which is similar to Figure~\ref{fig:symmetric-measures-1}.
  \begin{figure}[htb]
    \centering
\[\begin{tikzcd}
 Z\rtimes G_2 \arrow{r}{\chi\inverse}[swap]{s_{Z\rtimes G_2}} \arrow{d}[swap]{\chi}{r_{Z\rtimes G_2}}
 & Z \arrow{d}{\lambda}[swap]{\pi} \\
 Z \arrow{r}{\lambda}[swap] {\pi} 
& \Omega
 \end{tikzcd}\] 
\caption{}
\label{fig:pushig-measure-down}
  \end{figure}
Let $f \in \Contc(Z\rtimes G_2)$, then
\begin{align*}
\int_{Z\rtimes G_2} f\,\dd(m_u\circ \chi) &= \int_{Z} \int_{G_2} f((x,y), \gamma) \; \dd\chi_2^{s_X(x)}(\gamma)\, \dd m_u(x,y) \\
&= \int_Y\int_X \int_{G_2} f((x,y), \gamma) \; \dd\chi_2^{s_X(x)}(\gamma)\, \dd\alpha_{r_Y(y)}(x)\, \dd\beta_u(y).
\end{align*}
Change variable $((x,y),\gamma)\mapsto((x\gamma,\gamma\inverse y),\gamma\inverse)$. Then use the fact that the family measures $\alpha$ is $G_2$\nb-invariant and each measures in $\beta$ is $G_2$\nb-quasi-invariant to see that the previous term equals
\begin{align*}
\int_Y\int_X \int_{G_2} f((x\gamma,\gamma\inverse y),\gamma\inverse)\,\Delta_2(\gamma, \gamma\inverse y)
\;\,\dd\chi_2^{s_X(x)}(\gamma)\,\dd\alpha_{r_Y(y)}(x) \,\dd\beta_u(y).
\end{align*}
The function 
\[
\Delta\colon Z\rtimes G_2\to \R^+,\quad \Delta((x, y),\gamma)= \Delta_2(\gamma\inverse, y),
\]
is clearly continuous. Furthermore, $\Delta$ is an $\R^+_*$\nb-valued 1\nb-cocycle; for a composable pair $((x,y),\gamma),((x\gamma,\gamma\inverse y),\eta)\in Z\rtimes G_2$ a small routine computation shows that
\[
\Delta((x,y),\gamma)\Delta((x\gamma,\gamma\inverse y),\eta)=\Delta((x,y),\gamma\eta).
\] 
Thus $\Delta$ implements the \((Z\rtimes G_2,\chi)\)\nb-quasi-invariance of the measure $m_u$.

\noindent (ii): Since $ Z\rtimes G_2$ is a proper groupoid, we apply Proposition~\ref{prop:weak_equ_implies_equ} which gives a function $b\colon Z\to \R^+_*$ such that $bm_u$ is an invariant measure on $Z$ (with respect to \((Z\rtimes G_2,\chi)\)). The function $b$ also satisfies the relation $b\circ s_{Z\rtimes G_2}/b\circ r_{Z\rtimes G_2}=\Delta$, that is, $b(x\gamma,\gamma\inverse y)b(x,y)\inverse=\Delta((x,y),\gamma)$ for all $((x,y),\gamma)\in  Z\rtimes G_2$.
\end{proof}
\begin{remark}\label{rem:invariance-of-Delta}
For the cocycle $\Delta\colon Z\rtimes G_2\to \R^+_*$,
$\Delta((x,y),\gamma))=
\Delta_2(\gamma\inverse, y)$, we observe:
  \begin{enumerate}[label=\roman*), leftmargin=*]
  \item since $\Delta$ does not depend on $x$, $\Delta$ is $G_1$\nb-invariant;
  \item $\Delta_2$ is $G_3$\nb-invariant (\cite{Holkar2015Construction-of-Corr}*{Remark 2.5}). Hence 
    $\Delta(((x,y),\gamma)\gamma_3)=\Delta_2(\gamma\inverse, y\gamma_3)=\Delta_2(\gamma\inverse, y)=\Delta((x,y),\gamma)$ for all $((x,y),\gamma)\in  Z\rtimes G_2$ and appropriate $\gamma_3\in G_3$.
  \end{enumerate} 
 Thus $\Delta$ depends only on $\gamma$ and $[y]\in Y/G_3$.
\end{remark}

The function $b$ appearing in
Lemma~\ref{lemma:existance_of_composed_measure} can be computed explicitly. Let $p=\{p^z\}_{z\in Z}$ be a family of probability measures on $ Z\rtimes G_2$ as in Lemma~\ref{lemma:proper-gpd-has-prob-measures}. Then Propositions~\ref{prop:proper_gpd_first_cohom_zero} and~\ref{prop:weak_equ_implies_equ} give 
\begin{equation}
  \label{eq:formula-for-b}
  b(x,y) = (\exp\circ\,\underline{b})(x,y) = \exp\left( \int_{Z\rtimes G_2} \log \circ \Delta((x,y),\gamma)\; \dd p^{(x, y)}((x,y),\gamma) \right).
\end{equation}
This implies that $b$ is a continuous positive function on $Z$.
\begin{remark}\label{rem:invariance-of-b}
  \begin{enumerate}[label=\roman*)]
  \item The $G_1$\nb-invariance of $\Delta$ from Remark~\ref{rem:invariance-of-Delta} along with Equation~\ref{eq:formula-for-b} imply that $b$ is $G_1$\nb-invariant.
  \item The $G_3$\nb-invariance of $\Delta$ (Remark~\ref{rem:invariance-of-Delta} and Equation~\ref{eq:formula-for-b}) implies that $b$ is $G_3$-invariant. Indeed, for composable $((x,y), \gamma_3)\in Z\times G_3$
    \begin{multline*}
b(x,y\gamma_3)=b((x,y)\gamma_3) = \exp \left( \int \log \circ \Delta((x, y\gamma_3),\gamma)\; \dd p^{(x,y\gamma_3)}((x,y\gamma_3),\gamma) \right)\\=\exp \left( \int \log \circ \Delta((x, y\gamma_3),\gamma)\,F'((x, y),\gamma)\; \dd \chi_2^{r_Y(y\gamma_3)}(\gamma) \right) 
    \end{multline*}
 where $F'$ is a function as in Equation~\ref{eq:cut-off-function} for groupoid $Z\rtimes G_2$ used to get the family of probability measures $p$. The $G_3$ invariance of $\Delta$ and the fact that $r_Y(y\gamma_3)=r_Y(y)$ give that the previous term equals
\[\exp \left( \int \log \circ \Delta((x, y),\gamma)\,F'((x, y),\gamma)\; \dd \chi_2^{r_Y(y)}(\gamma) \right)=b(x,y).\]
The last equality is obtained from a computation similar to the one we started with, but in reverse order.
  \end{enumerate}
\end{remark}
\begin{remark}
  \label{rem:formula-for-mu-u}
Once we have $bm_u\circ\chi=bm_u\circ\chi\inverse$, (ii)
of Proposition~\ref{prop:pushing_measure} gives a measure $\mu_u$ on~$\Omega$ with $bm_u=\mu_u\circ\lambda$. And, as we shall see, $\{\mu_u\}_{u\in G_3}$ is
the required family of measures. For $f\in \Contc(\Omega)$
  \begin{multline}
    \label{eq:def-of-mu}
 \int_\Omega f\,\dd\mu_u=\int_Z (f\circ\pi)\cdot e\cdot\,b\,\dd m_u\\=
\int_Y\int_X f\circ\pi(x,y) e(x,y)\,b(x,y)\,\dd\alpha_{r_Y(y)}(x)\,\dd\beta_u(y)   
  \end{multline}
where $\pi\colon Z\to\Omega$ is the quotient map, and $e$ is the function on $Z$ with $\int e\circ s_{Z\rtimes G_2}\,\dd\chi^z = 1$ for all $z\in Z$. In the discussion that follows, the letter $e$ will always stand for such a function. Due to~{(iii)} of Proposition~\ref{prop:pushing_measure} the measure $\mu_u$ is
independent of the choice of the function $e$.
\end{remark} 

Recall that $\Omega$ is a $G_1$\nb-$G_3$\nb-bispace (Observation~\ref{obs:omega-is-G1-G3-bispace}) and the action of $G_3$ is proper (Lemma~\ref{lemma:composite-is-right-proper}). 
\begin{proposition}\label{prop:mu_u_is_measure}
The family of measures $\{\mu_u\}_{u\in{\base[G]}_3}$ is a\; $G_3$\nb-invariant continuous family of measures on $\Omega$ along the momentum map $s_\Omega$.
\end{proposition}
\begin{proof}
We check the invariance first and then check the continuity. Let $f\in
\Contc(\Omega)$ and $\gamma \in G_3$, then
\begin{align*}
&\int_\Omega f([x,y\gamma]) \; \dd\mu_{r_{G_3}(\gamma)}[x,y]\\
 &= \int_Y\int_X f([x, y\gamma]) e(x, y\gamma) b(x, y)\; \,\dd\alpha_{r_Y(y)}(x)\,\dd \beta_{r_{G_3}(\gamma')}(y).
\end{align*}

Change $y\gamma\to y$, then use the $G_3$\nb-invariance of the family
$\beta$ and that of the function $b$ to see that the last term in the above computation equals
\begin{align*}
&\int_Y\int_X f([x, y]) e(x,y) b(x, y)\,\dd\alpha_{r_Y(y)}(x)\,\dd \beta_{s_{G_3}(\gamma)}(y) \\
&=\int_Z f\cdot e \cdot\,b\,\dd m_{s_{G_3}(\gamma)}=\int_\Omega f[x,y] \; \dd\mu_{s_{G_3}(\gamma)}[x,y].
\end{align*}
Thus $\{\mu_u\}_{u\in\base[G_3]}$ is $G_3$\nb-invariant.

Now we check that $\mu$ is a continuous family of measures. Let
$M,\mu$ and $\Lambda$ denote the integration maps which the
families of measures $m,\mu$ and $\lambda$ induce between the
corresponding spaces of continuous compactly supported functions. Remark~\ref{rem:formula-for-mu-u} says that $M: \Contc(Z)\rightarrow \Contc({\base[G]}_3)$ is the
composite of  $\Contc(Z) \xrightarrow{\Lambda} \Contc(\Omega)
\xrightarrow{\mu} \Contc({\base[G]}_3)$, that is, Figure~\ref{fig:cont-of-mu} commutes: 
\begin{figure}[htb]
  \centering
  \begin{tikzcd}[row sep=tiny,scale=0.5]
                         			& \Contc(Z) \arrow{dl}[swap]{\Lambda}\arrow{dd}{M} \\
  \Contc(\Omega) \arrow{dr}[swap]{\mu}  &              \\
						&\Contc({\base[G]}_3)  .
\end{tikzcd}
\caption{}
\label{fig:cont-of-mu}
\end{figure}
Lemma~\ref{lemma:m_G_3_family} shows that $M$ is continuous,
\cite{Holkar2015Construction-of-Corr}*{Example 1.8} shows that $\Lambda$ is continuous and surjective. Hence $\mu$ is continuous.
\end{proof}
The family of measures $\mu$ on $\Omega$ is the required family of
measures for the composite correspondence. We still need to show that
each $\mu_u$ is $G_1$\nb-quasi-invariant. The following computation shows this quasi-invariance and also yields the adjoining function. Let $f\in \Contc(G_1 \times_{\base[G_1]}\Omega)$ and $u\in {\base[G]}_3$, then
\begin{align*}
&\int_\Omega\int_{G_1} f(\eta\inverse, [x,y]) \; \dd\chi_1^{r_\Omega([x,y])}(\eta)\, \dd\mu_u[x,y]\\
 &= \int_Y\int_X\int_{G_1} f(\eta\inverse, [x,y]) \,e(x,y)\, b(x, y)
   \,\dd\chi_1^{r_X(x)}(\eta)\, \dd\alpha_{r_Y(y)}(x) \, \dd\beta_u(y) .
\end{align*}

 Now we change variable $(\eta\inverse, [x,y])\mapsto (\eta, [\eta\inverse x,y])$. Then the $(G_1,\chi_1)$\nb-quasi-invariance of $\alpha$ changes
\[ \dd\chi_1^{r_X(x)}(\eta)\,\dd\alpha_{r_Y(y)}(x)\mapsto
\Delta_1(\eta, \eta\inverse x) \, \dd\chi_1^{r_X(x)}(\eta) \,\dd\alpha_{r_Y(y)}(x).
\]
We incorporate this change and continue the computation further:
\begin{multline*}
\textnormal{R.\,H.\,S.}=\int_Y\int_X\int_{G_1} f(\eta, [\eta\inverse x,y]) \,e(\eta\inverse x, y)\, b(\eta\inverse x, y)\\
\,\Delta_1(\eta, \eta\inverse x)\,\dd\chi_1^{r_X(x)}(\eta)\,\dd\alpha_{r_Y(y)}(x) \, \dd\beta_u(y)\\
=\int_Y\int_X\int_{G_1}f(\eta, [\eta\inverse x,y]) \,\frac{b(\eta\inverse x, y)}{b(x, y)}\, \Delta_1(\eta, \eta\inverse x)\\e(\eta\inverse x, y) b(x, y)\,\dd\chi_1^{r_X(x)}(\eta)\, \dd\alpha_{r_Y(y)}(x)\,\dd\beta_u(y).
\end{multline*}
We transfer the integration on $\Omega$ where the previous term equals
\[
\int_{\Omega}\int_{G_1} f(\eta, [\eta\inverse x,y]) \,\frac{b(\eta\inverse x, y)}{b(x, y)}\,\Delta_1(\eta, \eta\inverse x) \, \dd\chi_1^{r_\Omega([x,y])}(\eta)\, \dd\mu_u[x,y].
\]
Define $\Delta_{1,2} : G_1\ltimes\Omega \to \R^+_*$ by
\begin{equation}
\label{eq:def-delta}
\Delta_{1,2}(\eta, [x,y]) =b(\eta x,y)\inverse \Delta_1(\eta, x) b(x,y),
\end{equation}
then the above computation gives
\begin{multline*}
\int_\Omega\int_{G_1} f(\eta\inverse, [x,y]) \; \dd\chi_1(\eta) \;\dd\mu_u[x,y]\\ =  \int_\Omega\int_{G_1} f(\eta,[\eta\inverse x,y])\;\Delta_{1,2}(\eta, \eta\inverse [x,y]) \; \dd\chi_1(\eta) \;\dd\mu_u[x,y]  
\end{multline*}
for $u \in {\base[G]}_3$. To announce that $\mu_u$ is $G_1$\nb-quasi-invariant and $\Delta_{1,2}$ is the adjoining function, we must check that the function
$\Delta_{1,2}$ is well-defined which the next lemma does.
\begin{lemma}
  \label{lemma:delta-on-composite-is-well-def}
The function $\Delta_{1,2}$ defined in Equation~\eqref{eq:def-delta} is a well-defined $\R^+_*$\nb-valued continuous $1$\nb-cocycle on the groupoid $G_1\ltimes \Omega$.
\end{lemma}
\begin{proof}
  Let $(x\gamma,\gamma\inverse y)\in[x,y]$, then 
\[
\Delta_{1,2}(\eta\inverse, [x\gamma,\gamma\inverse y])=b(\eta\inverse x\gamma,\gamma\inverse y)\inverse \Delta_1(\eta\inverse, x\gamma) b(x\gamma,\gamma\inverse y).
\]
We multiply and divide the last term by $b(\eta\inverse x,y)\inverse b(x,y)$, and use the $G_2$\nb-invariance of~$\Delta_1$, then re-write the term as
\begin{align*}
b(\eta\inverse x,y)\inverse \Delta_1(\eta\inverse, x) b(x,y)\,\left(\frac{b(\eta\inverse x,y)}{ b(\eta\inverse x\gamma,\gamma\inverse y)} \,\frac {b(x\gamma,\gamma\inverse y)}{b(x,y)} \right)
\end{align*}
Now use the last claim in Lemma~\ref{lemma:existance_of_composed_measure} which relates $b$ and $\Delta_2$, use the definition of $\Delta_{1,2}$, and compute the above term further:
\begin{align*}
\Delta_{1,2}(\eta\inverse, [x,y])\,\left(\Delta_2(\gamma\inverse,
  y)\Delta_2(\gamma,\gamma\inverse y)\right)=\Delta_{1,2}(\eta\inverse, [x,y]).
\end{align*}
To get the equality above, observe that $(\gamma\inverse, y)\inverse=(\gamma,\gamma\inverse y)$ and use the fact that $\Delta_2$ is a homomorphism.

 Due to the continuity of $b$ and $\Delta_1$, the cocycle $\Delta_{1,2}$ is continuous. Using a computation as above, it can be checked that $\Delta_{1,2}$ is a groupoid homomorphism.
\end{proof}
\begin{proposition}
  \label{prop:mu-is-G-1-quasi-invariant}
The family of measures $\{\mu_u\}_{u\in{\base[G]}_3}$ is $G_1$\nb-quasi-invariant. The adjoining function for the quasi-invariance is given by Equation~\eqref{eq:def-delta}.
\end{proposition}
\begin{proof}
  Clear from the discussion above.
\end{proof}
\begin{definition}[Composition]
\label{def:composition}
Let
\begin{align*}
  (X, \alpha)&\colon (G_1,\chi_1)\allowbreak \rightarrow\allowbreak
                        (G_2,\chi_2) \textup{ and }\\
  (Y, \beta)&\colon (G_2,\chi_2) \rightarrow (G_3,\chi_3)
\end{align*}
be topological correspondences with $\Delta_1$ and $\Delta_2$ as the adjoining functions, respectively. A composite of these correspondence $(\Omega, \mu):(G_1,\chi_1)\rightarrow (G_3,\chi_3)$ is defined by:
\begin{enumerate}[label= {\roman*)}, leftmargin=*]
   \item the space $\Omega \defeq (X\times_{\base[G_2]}Y)/ G_2$;
   \item a family of measures $\mu = \{ \mu_u\}_{u\in{\base[G]}_3}$ such that
     \begin{enumerate}[leftmargin=*]
     \item let $\Delta\in \CC^1_{G_3}((X\times_{\base[G_2]}Y)\rtimes G_2, \R^+_*)$ be the 1\nb-cocycle $\Delta((x,y),\gamma)=\Delta_2(\gamma\inverse, y)$,
     \item let $b\in \CC^0_{G_3}((X\times_{\base[G_2]}Y)\rtimes G_2, \R^+_*)$ be a cochain with $d^0(b)=\Delta$;
     \item then $\mu$ disintegrate the family of measures $\{b(\alpha \times \beta_u)\}_{u\in{\base[G_3]}}$ on $X\times_{\base[G_2]}Y$ along the quotient map $\pi\colon X\times_{\base[G_2]}Y\to \Omega$ using $\lambda$, that is, $b(\alpha \times \beta_u)=\mu_u\circ\lambda$ for each $u\in \base[G_3]$.
     \end{enumerate}
\end{enumerate}
\end{definition}

\noindent  In Definition~\ref{def:composition}, $\CC^n_{G_3}((X\times_{\base[G_2]}Y)\rtimes G_2, \R^+_*)$ denotes the $n$\nb-th cochain group consisting of $G_3$\nb-invariant $\R^+_*$\nb-valued continuous cochains on groupoid $(X\times_{\base[G_2]}Y)\rtimes G_2$. For the composite $(\Omega,\mu)$ as above, the adjoining function $\Delta_{1,2}$ is given by Equation~\eqref{eq:def-delta}.

\begin{theorem}\label{thm:well-behaviour-of-composition}
Let $(X, \alpha)\co (G_1,\chi_1) \rightarrow (G_2,\chi_2)$ and $(Y, \beta)$ $\co (G_2,\chi_2) \rightarrow (G_3,\chi_3)$ be topological correspondences of locally compact groupoids with Haar systems. In addition, assume that $X$ and $Y$ are Hausdorff and second countable. Let $(\Omega, \mu)\co (G_1,\chi_1)\rightarrow (G_3,\chi_3)$ be a composite of the correspondences. Then $\Hils(\Omega) $ and $\Hils(X)\, \hat{\otimes}_{\Cst(G_2,\chi_2)} \Hils(Y)$ are isomorphic
$\Cst$\nb-correspondences from $\Cst(G_1,\chi_1)$ to $\Cst(G_3,\chi_3)$.
\end{theorem}
\begin{proof}
 The symbols $Z$, $\Omega$ and the families of measures $m,\lambda$ and $\mu$ continue to have the same meaning as in the earlier discussion. Let $b$ be a fixed zeroth cocycle on $Z\rtimes G_2$ with $\Delta=d^0(b)$ as in Definition~\ref{def:composition}. In the calculations below, the subscripts to $\inpro{}{}$ indicate the Hilbert module on which the inner product is defined. We write $\Hils(X)\hat\otimes \Hils(Y)$ instead of $\Hils(X)\hat{\otimes}_{\Cst(G_2,\chi_2)} \Hils(Y)$ in this proof to reduce the complexity in writing.

Recall the process of composing two $\Cst$\nb-correspondences in Section~\ref{rev} on page~\pageref{page:compo-of-hilm}. We know that $\Contc(X)\subseteq\Hils(X)$ and $ \Contc(Y)\subseteq\Hils(Y)$ are, respectively, pre-Hilbert  $\Cst(G_2,\chi_2)$ and$\Cst(G_3,\chi_3)$\nb-modules. Due to this density, the image of $\Contc(X)\mathbin{\otimes_{C}}\Contc(Y)\to \Hils(X)\mathbin{\hat\otimes}\Hils(Y)$ under the obvious map is dense. Recall from the same discussion that the obvious map sends $f\otimes g\in \Contc(X)\otimes_{\C}\Contc(Y)$ to its equivalence class in $(\Contc(X)\otimes_{\C}\Contc(Y))/(\Contc(X)\otimes_{\C}\Contc(Y))\cap N$ where $N\subseteq \Hils(X)\mathbin{\hat\otimes}\Hils(Y)$ is the subspace of vectors of zero norm. The norm on $\Contc(X)\otimes_{\C} \Contc(Y)$ induced by the inner product
\begin{equation}
  \label{eq:inn-pro-on-cst-comp}
\inpro{f\otimes g}{f'\otimes g'}_{\Hils(X)\mathbin{\hat\otimes}\Hils(Y)}\defeq \inpro{g}{\inpro{f}{f'}_{\Hils(X)}g'}_{\Hils(Y)}
\end{equation}
for $f\otimes g, f'\otimes g'\in\in \Contc(X)\otimes_{\C} \Contc(Y)$.

Thus when equipped with the inner product in Equation~\ref{eq:inn-pro-on-cst-comp}, the pre-Hilbert $\Cst(G_3,\chi_3)$\nb-module $\Contc(X)\otimes_{\C}\Contc(Y)$ completes to the Hilbert $\Cst(G_3,\chi_3)$\nb-module $\Hils(X)\mathbin{\hat\otimes}\Hils(Y)$.

On the other hand, $\Contc(\Omega)\subseteq\Hils(\Omega)$ is a pre-Hilbert $\Cst(G_3,\chi_3)$\nb-module. We define an inner product preserving $\Contc(G_3)$\nb-module map $\Lambda'\colon  \Contc(X)\otimes \Contc(Y)\to \Contc(\Omega)$ which has a dense image. Then $\Lambda'$ is an inner product preserving map of pre-Hilbert $\Cst(G_3,\chi_3)$\nb-modules and image of $\Lambda'$ is dense. Thus $\Lambda'$ extends to an isomorphism of Hilbert $\Cst(G_3,\chi_3)$\nb-modules $\Hils(X)\hat\otimes\Hils(Y)\to \Hils(\Omega)$.

 After showing that $\Lambda'$ is an isometry of $\Contc(G_3)$\nb-modules and has a dense image, we show that $\Lambda'$ intertwines the representations of $\Cst(G_1,\chi_1)$ on $\Hils(X)\hat\otimes\Hils(Y)$ and $\Hils(\Omega)$ which completes the proof that $\Lambda'$ induces an isomorphism  $\Hils(X)\hat\otimes\Hils(Y)\to \Hils(\Omega)$ of $\Cst$\nb-correspondences. 

Again, due to the density arguments as above, it is enough to show that $\Lambda'$ intertwines the representations of $\Contc(G_1)$ on $\Contc(X)\otimes_{\C} \Contc(Y)$ and $\Contc(\Omega)$; and this is what we show. Thus the proof is divided into two parts, the first part proves the isomorphism of Hilbert modules and the other the isomorphism of the representations.

The strategy of the proof is explained and we start the proof by defining $\Lambda'$. Map $f\otimes g\in \Contc(X)\otimes_{\C} \Contc(Y)$ to $(f\otimes g)|_Z\in \Contc(Z)$ where $(f\otimes g)|_Z(x,y)=f(x)g(y)$ for $(x,y)\in Z$. Then the Stone-Weierstra{\ss} theorem gives that the set $\{(f\otimes g)\lvert_{Z}: f\otimes g\in \Contc(X)\otimes_{\C} \Contc(Y)\}\subseteq \Contc(Z)$ is dense. Define $\Lambda': \Contc(X)\otimes_{\C} \Contc(Y) \rightarrow \Contc(\Omega)$ by 
\begin{align*}
\Lambda'(f\otimes g)[x,y]
 &=\Lambda((f\otimes g)\lvert_{Z}b^{-1/2})[x,y]\\
 &= \int_{G_2} (f\otimes g)|_Z(x\gamma, \gamma\inverse y)  \,b^{-1/2}(x\gamma,\gamma\inverse y) \,\dd\chi_2^{s_X(x)}(\gamma)  
\end{align*}
 where $f\otimes g \in \Contc(X)\otimes_{\C} \Contc(Y)$.
Since $b$ is a positive function, the multiplication by $b^{-1/2}$ is an isomorphism from $\Contc(Z)$ to itself. As $\lambda$ is a continuous family of measure with full support, $\Lambda\colon \Contc(Z)\to \Contc(\Omega)$ is surjection. Thus the composite $\Lambda'\co \Contc(X)\otimes_{\C} \Contc(Y)\xrightarrow{f\otimes g\mapsto(f\otimes g)\lvert_{Z}} \Contc(Z) \xrightarrow{\textnormal{multiplication by }b^{-1/2}}  \Contc(Z) \xrightarrow{\Lambda} \Contc(\Omega)$ is a continuous and has dense image.

Let $z\in\C, f,f'\in \Contc(X)$ and $g,g'\in \Contc(Y)$. Then it is straightforward computation to check that $\Lambda'(z f\otimes g+f'\otimes g')=z\Lambda'(f\otimes g)+\Lambda'(f'\otimes g')$. Furthermore, if $\psi\in \Contc(G_3)$, then a computation using Fubini's theorem shows that $\Lambda'((f\otimes g)\psi)=\Lambda'(f\otimes g)\psi$. Thus $\Lambda'$ is a homomorphism of $\Contc(G_3)$\nb-modules.

\paragraph{{\bfseries The isomorphism of the Hilbert modules:}}
In this part, we show that $\Lambda'$ preserves $\Cst(G_3,\chi_3)$\nb-valued inner products. Let $f\otimes g \in \Contc(X)\otimes_{\C} \Contc(Y)$ and $\underline{\gamma} \in G_3$, then
\begin{align*}
&\inpro{f\otimes g}{f \otimes g}_{\Hils(X)\mathbb{\hat\otimes}\Hils(Y)} (\underline{\gamma})\defeq \inpro{g}{\inpro{f}{f}_{\Hils(X)} g}_{\Hils(Y)}(\underline{\gamma}) \notag\\
&= \int_Y \overline{g(y)}\;\; (\inpro{f}{f}_{\Hils(X)} \; g)(y\underline{\gamma}) 
\;\; \dd\beta_{r_{G_3}(\underline{\gamma})}(y)\notag\\
&=\int_Y  \int_{G_2}\overline{g(y)}\; \langle f, f \rangle_{\Hils(X)}(\gamma)\; g(\gamma\inverse y\underline{\gamma})\notag
\Delta^{1/2}_2(\gamma, \gamma\inverse y\underline{\gamma}) \;\dd\chi_2^{r_Y(y)}(\gamma)\dd\beta_{r_{G_3}(\underline{\gamma})}(y)\notag\\
&=\int_Y\int_{G_2} \overline{g(y)}\left( \int_X \overline{f(x)} f (x\gamma) \,\dd\alpha_{r_{G_2}(\gamma)}(x) \right)\notag\\
&\qquad\quad g(\gamma\inverse y\underline{\gamma})\;\Delta^{1/2}_2(\gamma, \gamma\inverse y\underline{\gamma}) \;\dd\chi_2^{r_Y(y)}(\gamma) \;\dd\beta_{r_{G_3}(\underline{\gamma})}(y) \notag
\end{align*}
We rearrange the functions, note that $r_{G_2}(\gamma)=r_Y(y)$ and write the last term as
\begin{multline}
\label{eq:left-inn-pro-value}
\int_Y\int_{G_2}\int_X \overline{f(x)}\, \overline{g(y)} \; f (x\gamma) g(\gamma\inverse y\underline{\gamma})\\
\Delta^{1/2}_2(\gamma, \gamma\inverse y\underline{\gamma}) \;\dd\alpha_{r_Y(y)}(x) \;\dd\chi_2^{r_Y(y)}(\gamma) \;\dd\beta_{r_{G_3}(\underline{\gamma})}(y). 
\end{multline}

Now we calculate the norm of $\Lambda'(f\otimes g)\in \Contc(\Omega)$:
\begin{multline*}
\inpro{\Lambda'(f\otimes g)}{\Lambda'(f \otimes g)}_{\Hils(\Omega)}(\underline{\gamma})\\
\defeq \int_\Omega \overline{\Lambda'(f\otimes g)[x,y]}\; \Lambda'(f\otimes g)[x,y\underline{\gamma}] \; \dd\mu_{r_{G_3}(\underline{\gamma})}[x,y].
\end{multline*}
We plug the value of the first $\Lambda'(f\otimes g)$ and continue computing further:
\begin{align*}
  &\int_\Omega  \left(\int_{G_2}\overline{f(x\gamma_*)} \overline{g(\gamma_*\inverse y)} b^{-1/2}(x\gamma_*, \gamma_*\inverse y)\;\dd\chi_2^{r_Y(y)}(\gamma_*)\right)
\Lambda'(f\otimes g)[x,y\underline{\gamma}] \, \dd\mu_{r_{G_3}(\underline{\gamma})}[x,y] \\
&= \int_\Omega\int_{G_2} \overline{f(x\gamma_*)} \overline{g(\gamma_*\inverse y)} b^{-1/2}(x\gamma_*, \gamma_*\inverse y) 
\Lambda'(f\otimes g)[x,y\underline{\gamma}]\,
\dd\chi_2^{r_Y(y)}(\gamma_*)\; \dd\mu_{r_{G_3}(\underline{\gamma})}[x,y]\\
&=  \int_Y\int_X \overline{f(x)}\,\overline{g(y)} b^{-1/2}(x,y) \Lambda'(f\otimes g)[x,y\underline{\gamma}]b(x,y)\,
\dd\alpha_{r_Y(y)}(x) \;\dd\beta_{r_{G_3}(\underline{\gamma})}(y).
\end{align*}
The last equality above is due to Remark~\ref{rem:formula-for-mu-u}, which says that
 \[\dd\chi_2^{r_Y(y)}(\gamma_*) \;\dd\mu_{r_{G_3}(\underline{\gamma})}[x,y] = b(x, y) \; \dd\alpha_{r_Y(y)}(x) \;\dd\beta_{r_{G_3}(\underline{\gamma})}(y).\]
We process the function $b$ in the previous term,  plug in the value of $\Lambda'(f\otimes g)$ and 
compute further,
\begin{align*}
\text{L.\,H.\,S.}&=\int_Y\int_{X} \overline{f(x)}\,\overline{g(y)} 
\left(\int_{G_2} f(x\gamma) g(\gamma\inverse y\underline{\gamma})b^{-1/2}(x\gamma,\gamma\inverse y\underline{\gamma})\;\dd\chi_2^{r_Y(y)}(\gamma)\right)\\
            &\qquad b^{1/2}(x,y) \;
\dd\alpha_{r_Y(y)}(x) \;\dd\beta_{r_{G_3}(\underline{\gamma})}(y) \notag\\
&= \int_Y\int_{X} \int_{G_2} \overline{f(x)}\, \overline{g(y)} f(x\gamma)
  g(\gamma\inverse y\underline{\gamma})\notag\\
            &\qquad \left(\frac{b(x,y)}{b(x\gamma,\gamma\inverse y\underline{\gamma})} \right)^{1/2} \dd\chi_2^{r_Y(y)}(\gamma)\;\dd\alpha_{r_Y(y)}(x)\;\dd\beta_{r_{G_3}(\underline{\gamma})}(y).\notag
\end{align*}
First we use the $G_3$\nb-invariance of $b$ (Remark~\ref{rem:invariance-of-b}) to write $b(x,y)=b(x,y\underline{\gamma})$. Then we use Lemma~\ref{prop:weak_equ_implies_equ} to relate the factors of $b$ and get a factor of $\Delta$ which can be written in terms of $\Delta_2$ using Remark~\ref{rem:invariance-of-Delta}. At the end of these computations, the last term of the previous becomes
\begin{multline*}
\int_Y\int_{X} \int_{G_2} \left(\overline{f(x)}\, \overline{g(y)} f(x\gamma) g(\gamma\inverse y\underline{\gamma})\right)\\ {\Delta_2}^{1/2}(\gamma, \gamma\inverse y\underline{\gamma})\;\dd\chi_2^{r_Y(y)}(\gamma)\,\dd\alpha_{r_Y(y)}(x)\,\dd\beta_{r_{G_3}(\underline{\gamma})}(y).
\end{multline*}
Finally, we apply Fubini's Theorem to $\chi_2^{r_Y(y)}$ and $\alpha_{r_Y(y)}$ to get
\begin{multline}\label{eq:right-inn-pro-value}
\inpro{\Lambda'(f\otimes g)}{\Lambda'(f \times g)}_{\Hils(\Omega)}(\underline{\gamma})
\\= \int_Y\int_{G_2}\int_{X} \left(\overline{f(x)}\, \overline{g(y)} f(x\gamma) g(\gamma\inverse y\underline{\gamma})\right)\\{\Delta_2}^{1/2}(\gamma, \gamma\inverse y\underline{\gamma}) \;\dd\alpha_{r_Y(y)}(x)\;\dd\chi_2^{r_Y(y)}(\gamma) \;\dd\beta_{r_{G_3}(\underline{\gamma})}(y).
\end{multline}

Comparing the values of both inner products, that is, Equation~\ref{eq:left-inn-pro-value} and~\ref{eq:right-inn-pro-value}, we conclude that
\begin{equation*}
\inpro{f\otimes g}{f \otimes g}_{\Hils(X)\mathbin{\hat{\otimes}} \Hils(Y)} = \inpro{\Lambda'(f\otimes g)}{\Lambda'(f \otimes g)}_{\Hils(\Omega)}.
\end{equation*}
\paragraph{{\bfseries The isomorphism of representations:}}

Denote the actions of $\Cst(G_1,\chi_1)$ on $\Hils(X)\,
\hat{\otimes} \Hils(Y)$ and $\Hils(\Omega)$ by $\rho_1$ and
$\rho_2$, respectively, that is, $\rho_1\co \Cst(G_1,\chi_1)\to
\Bound(\Hils(X)\, \hat{\otimes}_{\Cst(G_2)} \Hils(Y))$ and $\rho_2\co
\Cst(G_1,\chi_1)\to \Bound( \Hils(\Omega))$ are the nondegenerate \Star{}representations that give the
$\Cst$\nb-correspondences from $\Cst(G_1,\chi_1)$ to
$\Cst(G_3,\chi_3)$. Now we show that $\Lambda'$ intertwines $\rho_1$ and $\rho_2$.

Let $\Delta_{1,2}$ be the adjoining function of $(\Omega,\mu)$ which is given by Equation~\ref{eq:def-delta}. Let $\phi\in \Contc(G_1)$ and $f\otimes g\in \Contc(X)\otimes_{\C} \Contc(Y)$, then
\begin{align}
&(\rho_2(\phi)\Lambda')(f\otimes g)[x,y]
  = (\phi*\Lambda'(f\otimes g))[x,y]\notag\\
&=\int_{G_1} \phi(\eta) \Lambda'(f\otimes g))[\eta\inverse x, y]\, \Delta_{1,2}^{1/2}(\eta,[\eta\inverse x, y])\,\dd\chi_1^{r_X(x)}(\eta)\notag\\
&=\int_{G_1}\int_{G_2} \phi(\eta) f(\eta\inverse x\gamma) g(\gamma\inverse y)\,b^{-1/2}(\eta\inverse x\gamma,\gamma\inverse y)\label{eq:iso-of-repre-composite-corr}\\
&\qquad\Delta_{1,2}^{1/2}(\eta,[\eta\inverse x, y])\ \,\dd\chi_2^{s_X(x)}(\gamma)\,\dd\chi_1^{r_X(x)}(\eta). \notag
\end{align}
Lemma~\ref{lemma:delta-on-composite-is-well-def} and Equation~\eqref{eq:def-delta} allows us to write
 \[
\Delta_{1,2}(\eta, [\eta\inverse x,y])=\Delta_{1,2}(\eta,[\eta\inverse x\gamma,\gamma\inverse y])=\Delta_1(\eta,\eta\inverse x\gamma)\,\frac{b(\eta\inverse x\gamma,\gamma\inverse y)}{b(x\gamma,\gamma\inverse y)}.
\] 
Substitute this value of $\Delta_{1,2}(\eta, [\eta\inverse x,y])$ in Equation~\ref{eq:iso-of-repre-composite-corr}. Then apply Fubini's theorem and continuing computing further: 
\begin{align*}
  &\int_{G_2}\left(\int_{G_1} \phi(\eta) f(\eta\inverse x\gamma) \,\Delta_1^{1/2}(\eta,\eta\inverse x\gamma)\,\dd\chi_1^{r_X(x)}(\eta)\right)\\
&\qquad g(\gamma\inverse y)\,\,b^{-1/2}( x\gamma,\gamma\inverse y)\,\dd\chi_2^{s_X(x)}(\gamma)\\
&=\int_{G_2} (\phi *f)(x\gamma) g(\gamma\inverse y)\,b^{-1/2}( x\gamma,\gamma\inverse y) \,\dd\chi_2^{s_X(x)}(\gamma)\\
&=\Lambda'((\phi*f)\otimes g)[x,y]=\Lambda'(\rho_1(\phi)(f\otimes g))[x,y].\qedhere
\end{align*}
\end{proof}

\section{Examples}
\begin{example}
  \label{exa:cont-function-as-corr-2}
Let $X, Y$ and $Z$ be locally compact Hausdorff spaces and let $f\co X\to Y$ and $g\colon Y\to Z$ be a continuous functions. Then~\cite{Holkar2015Construction-of-Corr}*{Example 3.1} shows that $(X,\delta_X)$ is a topological correspondence from $Y$ to~$X$ and $(Y,\delta_Y)$ is the one from $Y$ to~$Z$. Here $\delta_X=\{\delta_x\}_{x\in X}$ is the family of measures consisting of point masses along the identity map $X\to X$. Similar is the meaning of $\delta_Y$. The constant function $1$ is the adjoining function for both correspondences.

The space involved the composite of $(Y,\delta_Y)$ and $(X,\delta_X)$ is $(Y\times_{\Id_Y, Y, f} X)\homeo X$, and the homeomorphism $(Y\times_{\Id_Y, Y, f} X)\to X$ is implemented by the function $(f(x),x)\mapsto x$. The inverse of this function is $x\mapsto (f(x),x)$. The left momentum map $Y\times_{\Id_Y, Y, f}X\to Z$ is $(f(x),x)\mapsto g(f(x))$ which we identify with $g\circ f\colon X\to Z$. Thus the composite of the topological correspondences related to continuous maps is same as the topological correspondence related to the composite of the maps. Reader may check the $\Cst$\nb-algebraic counterpart of this example agree with Theorem~\ref{thm:well-behaviour-of-composition}.
\end{example}
\begin{example}
  \label{exm:topological-quiver-top-corr-2}
Let $V,W, X, Y$ and $Z$ be locally compact Hausdorff spaces and let $f\colon X\to Z, g\colon X\to Y, k\colon V\to Y$ and $l\colon V\to W$ be continuous maps. Let $\lambda_1$ and $\lambda_2$ be continuous families of measures along $g$ and $l$, respectively (See Figure~\ref{fig:example-quiver} on page~\pageref{fig:example-quiver}). Then $(X,\lambda_1)$ is a topological correspondence from $Z$ to $Y$ and $(V,\lambda_2)$ is one from $Y$ to $W$ (\cite{Holkar2015Construction-of-Corr}*{Example 3.3}). 
\begin{figure}[htb]
  \centering
\[
\begin{tikzcd}[column sep=small]
 & X \arrow{dl}[swap]{f}\arrow{dr}{g}[swap]{\lambda_1}&  &V \arrow{dl}[swap]{k}\arrow{dr}{l}[swap]{\lambda_2}\\
  Z &                        & Y  &                & W
\end{tikzcd}
\]  
\caption{}
\label{fig:example-quiver}
\end{figure}
The composite correspondence is $(X\times_{g,Y,k}V, \lambda_1\circ\lambda_2)$ where $(\lambda_1\circ\lambda_2)_w$ is defined by 
\[
\int_{X\times_{g,Y,k}V} f\,\dd(\lambda_1\circ\lambda_2)_w=\int_V\int_X f(x,v)\,\dd{\lambda_1}_{k(v)}(x)\,\dd{\lambda_2}_{w}(v)
\] for $w\in W$ and $f\in \Contc(X\times_{g,Y,k}V)$. Note that in this example $\lambda_1\circ\lambda_2$ is the family of measures $m$  in Lemma~\ref{lemma:m_G_3_family} and, since there are only the trivial actions, it is same as the family of measures $\mu$ in Proposition~\ref{prop:mu_u_is_measure}.
\end{example}
\begin{example}
  \label{exa:gp-homo-as-corr}
Let $G,H$ and $K$ be locally compact groups,  $\psi\colon K\to H$ and $\phi\co H\to G$ continuous homomorphisms. Let $\alpha, \beta$ and $\lambda$ be the Haar measures on $G, H$ and $K$, respectively. Then $(G,\alpha\inverse)$ is a correspondence from $(H,\beta)$ to $(G,\alpha)$, and $(H,\beta\inverse)$ is one from $(K,\lambda)$ to $(H,\beta)$ (\cite{Holkar2015Construction-of-Corr}*{Example 3.4}). Let $\delta_G, \delta_H$ and $\delta_K$ denote the modular functions of $G,H$ and $K$, respectively. Then $\frac {\delta_G\circ\phi}{\delta_H}$ and $\frac {\delta_H\circ\phi}{\delta_K}$ are the adjoining function for these correspondences, respectively. 

The $K$-$G$\nb-bispace in composite of these correspondences is $(H\times G)/H\homeo G$. The map $a\colon\gamma\mapsto[e_H,\gamma]$, $G\to (H\times G)/H,$ gives the homeomorphism where $e_H$ is the unit in $H$. The inverse of this map $a\inverse$ is $a\inverse\colon[\eta, \gamma]\mapsto \phi(\eta)\gamma$, $(H\times G)/H\to G$.

We figure out the action of $K$ on this $K$\nb-$G$-bispace: if $\kappa\in K$ then $\kappa\gamma= a\inverse(\kappa[e_H,\gamma])\defeq a\inverse([\psi(\kappa),\gamma])=\phi(\psi(\kappa))\gamma$. Thus $K$ acts on $G$ via the homomorphism $\phi\circ\psi\colon K\to G$. Similarly, the right action of $G$ on the composite space $(H\times G)/H\homeo G$ is identified with the right multiplication action of $G$ on itself. A computation as in~\cite{Holkar2015Construction-of-Corr}*{Example 3.4} gives that $\frac {\delta_G\circ\phi\circ\psi}{\delta_K}$ is the adjoining function for the composite correspondence.

This shows that the composite of $(H,\beta\inverse)$ and $(G,\alpha\inverse)$ is same as the correspondence associated with the homomorphism $\phi\circ\psi\colon K\to G$.
\end{example}
\begin{example}\label{ex:subgroup_action-2}
Let $(G,\alpha), (H,\beta)$ and $(K,\lambda)$ be locally compact groups with Haar measures, and let $\phi\colon H\to G$ and $\psi\co K\to G$ be continuous homomorphisms. Assume the $\psi$ is a proper map. Then $\phi$ gives a correspondences $(G,\alpha\inverse)$ from $(H,\beta)$ to $(G,\alpha)$ as in~\cite{Holkar2015Construction-of-Corr}*{Example 3.4} and $\psi$ gives a correspondence $(G,\alpha\inverse)$ from $(G,\alpha)$ to $(K,\lambda)$ as in~\cite{Holkar2015Construction-of-Corr}*{Example 3.5}. The adjoining function of the topological correspondence associated with $\psi$ is the constant function $1$.

The composite of these correspondences is a correspondence $(H,\beta)\to (K,\lambda)$.  One the similar lines of Example~\ref{exa:gp-homo-as-corr}, one may show that the space involved in the composite is homeomorphic to $G$, the actions of $H$ and $K$ are identified with the left and right multiplication via $\phi$ and $\psi$, the $K$\nb-invariant family of measures on $G$ is $\alpha\inverse$. From~\cite{Holkar2015Construction-of-Corr}*{Example 3.4} we know that the $\alpha\inverse$ is $(H,\beta)$\nb-quasi-invariant and the function $\frac{\delta_G\circ\phi}{\delta_H}$ is the cocycle involved in the quasi\nb-invariance. Hence $\frac{\delta_G\circ\phi}{\delta_H}$ is the adjoining function for the composite.

An interesting situation is when $H,K\subseteq G$ are closed subgroups, and $\phi$ and $\psi$ are the inclusion maps. Then the composite correspondence from $(H,\beta)$ to $(K,\lambda)$ is $(G,\alpha\inverse)$ where $G$ is made into an $H$\nb-$K$ bispace using the left and right multiplication actions, respectively. The adjoining function in this case is $\frac{\delta_G}{\delta_H}$. 
\end{example}
\begin{example}\label{exm:Stadler-Ouchi-correspondence-2}
Example~{3.7} in~\cite{Holkar2015Construction-of-Corr} shows that the correspondences defined by Macho Stadler and  O'uchi in~\cite{Stadler-Ouchi1999Gpd-correspondences} are topological correspondences. The same example shows that the topological correspondences defined by Tu in~\cite{Tu2004NonHausdorff-gpd-proper-actions-and-K}*{Proposition 7.5} are also topological correspondences provided that the spaces of the units of the groupoids are Hausdorff. The composition of correspondences of Macho Stadler and O'uchi defined by Tu (\cite{Tu2004NonHausdorff-gpd-proper-actions-and-K}) is same as the composition we define. 

Recall from~\cite{Stadler-Ouchi1999Gpd-correspondences}*{Definition 1.2} or~\cite{Tu2004NonHausdorff-gpd-proper-actions-and-K}*{Definition 7.3} that a correspondence $(G_1,\chi_1)\to (G_2,\chi_2)$ is a $G_1$\nb-$G_2$-bispace, and the actions and the quotient $G_1\7 X$ satisfy certain conditions. Since the correspondences of Macho Stadler and O'uchi or Tu do not involve explicit families of measures, the construction of the composite in this is purely topological. If $Y$ a correspondence in there sense from $(G_2,\chi_2)$ to $(G_3,\chi_3)$, then Tu shows~\cite{Tu2004NonHausdorff-gpd-proper-actions-and-K} the space $\Omega$ in Definition~\ref{def:composition} the composite. 
\end{example}
\begin{example}
\label{exa:Buneci-Stachura-exm}
  Example~{3.10} in~\cite{Holkar2015Construction-of-Corr} shows that the generalized morphisms defined by Buneci and Stachura are topological correspondences in our sense. Though it is not as straightforward as in Example~\ref{exm:Stadler-Ouchi-correspondence-2} above, but it may be checked that the composition of the generalized morphisms of Buneci and Stachura defined in~\cite{Holkar2015Construction-of-Corr}*{Section 2.2} match our definition of composition.
\end{example}
\begin{example}
\label{exa:topological-induction}
Let $G$ be a locally compact group, let $H$ and $K$ be closed subgroups of $G$, and let  $\alpha, \beta$ and $\lambda$ be the Haar measures on $G, H$ and $K$, respectively. Let $\delta_G$ and $\delta_H $ be the modular functions of $G$ and $H$, respectively. Then $\big(G, \alpha^{-1}\big)$ is a correspondence from $H$ to $ K$ with $ \frac{\delta_G}{\delta_H}$ as the adjoining function, see in Example~\ref{ex:subgroup_action-2}.

 Let $X$ be a left $K$-space carrying a strongly $(K,\lambda)$\nb-quasi-invariant measure $\kappa$, that is, $\kappa$ is a $(K,\lambda)$\nb-quasi-invariant measure on $X$ and the Radon-Nikodym derivative for the quasi-invariance, say $\Delta\colon K\ltimes X\to \R^+_*$, is a continuous function. Then $(X, \kappa)$ is a correspondence from $K$ to $\textup{Pt}$ with $\Delta$ as adjoining function. Here $\textup{Pt}$ stands for the trivial group(oid) which consists of the unit only.

We discuss the composite of these two correspondences. The space in the composite is $(G\times X)/K$, which we denote by $Z$. In this example, writing the measure $\nu$ on $Z$ concretely is not always possible. However, when $(X,\kappa)=(K,\lambda)$, we get  $Z\homeo G$ and $\nu=\alpha\inverse$.

The correspondence $(X,\kappa)$ gives a representation of $K$ on $\Ltwo(X,\kappa)$ and the composite correspondence is the representation of $H$ \emph{induced} by this representation of $K$.

\end{example}

\paragraph{\textbf{Acknowledgement:}} I am grateful to Jean Renault and Ralf Meyer for their guidance and many fruitful discussions.


\begin{bibdiv}
  \begin{biblist}
  \bib{Anantharaman-Renault2000Amenable-gpd}{book}{
   author={Anantharaman-Delaroche, C.},
   author={Renault, J.},
   title={Amenable groupoids},
   series={Monographies de L'Enseignement Math\'ematique [Monographs of
   L'Enseignement Math\'ematique]},
   volume={36},
   note={With a foreword by Georges Skandalis and Appendix B by E. Germain},
   publisher={L'Enseignement Math\'ematique, Geneva},
   date={2000},
   pages={196},
   isbn={2-940264-01-5},
   review={\MR{1799683 (2001m:22005)}},
}

\bib{Bourbaki2004Integration-II-EN}{book}{
   author={Bourbaki, Nicolas},
   title={Integration. II. Chapters 7--9},
   series={Elements of Mathematics (Berlin)},
   note={Translated from the 1963 and 1969 French originals by Sterling K.
   Berberian},
   publisher={Springer-Verlag, Berlin},
   date={2004},
   pages={viii+326},
   isbn={3-540-20585-3},
   review={\MR{2098271 (2005f:28001)}},
}

\bib{Buneci-Stachura2005Morphisms-of-lc-gpd}{article}{
  author={Buneci, M\u{a}d\u{a}lina Roxana},
  author={Stachura, Piotr},
  title={Morphisms of locally compact groupoids endowed with Haar systems},
  status={eprint},
  date={2005},
  eprint={arxiv:0511613},
}

\bib{Folland1995Harmonic-analysis-book}{book}{
   author={Folland, Gerald B.},
   title={A course in abstract harmonic analysis},
   series={Studies in Advanced Mathematics},
   publisher={CRC Press, Boca Raton, FL},
   date={1995},
   pages={x+276},
   isbn={0-8493-8490-7},
   review={\MR{1397028 (98c:43001)}},
}

\bib{Holkar2015Construction-of-Corr}{article}{
  title={Topological construction of $C^*$-correspondences
    for groupoid $C^*$-algebras},
  author={Holkar, Rohit Dilip},
  date={2015},
  archivePrefix={arXiv},
  eprint={arxiv:1510.07534},
  primaryClass={math.OA math.FA},
}

\bib{Lance1995Hilbert-modules}{book}{
   author={Lance, E. C.},
   title={Hilbert $C^*$-modules},
   series={London Mathematical Society Lecture Note Series},
   volume={210},
   note={A toolkit for operator algebraists},
   publisher={Cambridge University Press, Cambridge},
   date={1995},
   pages={x+130},
   isbn={0-521-47910-X},
   review={\MR{1325694 (96k:46100)}},
   doi={10.1017/CBO9780511526206},
}

\bib{Stadler-Ouchi1999Gpd-correspondences}{article}{
  author={Macho Stadler, Marta},
  author={O'uchi, Moto},
  title={Correspondence of groupoid $C^*$-algebras},
  journal={J. Operator Theory},
  volume={42},
  date={1999},
  number={1},
  pages={103--119},
  issn={0379-4024},
  eprint={www.mathjournals.org/jot/1999-042-001/1999-042-001-005.pdf},
}

\bib{Tu2004NonHausdorff-gpd-proper-actions-and-K}{article}{
   author={Tu, Jean-Louis},
   title={Non-Hausdorff groupoids, proper actions and $K$-theory},
   journal={Doc. Math.},
   volume={9},
   date={2004},
   pages={565--597 (electronic)},
   issn={1431-0635},
   review={\MR{2117427 (2005h:22004)}},
}
  \end{biblist}
\end{bibdiv}
\end{document}